\documentclass{amsart}

\usepackage{amsmath}
\usepackage{amsthm}
\usepackage{amssymb}
\usepackage{color}
\usepackage{hyperref}
\usepackage{amsfonts,graphics,amsthm,amsfonts,amscd,latexsym}
\usepackage{epsfig}
\usepackage{flafter}
\usepackage{mathtools}
\usepackage{comment}
\usepackage{stmaryrd}

\usepackage{mathabx,epsfig}

\hypersetup{
    colorlinks=true,    
    linkcolor=blue,          
    citecolor=blue,      
    filecolor=blue,      
    urlcolor=blue           
}
\usepackage{tikz}
\usetikzlibrary{graphs,positioning,arrows,shapes.misc,decorations.pathmorphing}

\tikzset{
    >=stealth,
    every picture/.style={thick},
    graphs/every graph/.style={empty nodes},
}

\tikzstyle{vertex}=[
    draw,
    circle,
    fill=black,
    inner sep=1pt,
    minimum width=5pt,
]
\usepackage[position=top]{subfig}

\calclayout

\usetikzlibrary{decorations.pathmorphing}
\tikzstyle{printersafe}=[decoration={snake,amplitude=0pt}]

\newcommand{\rank}{\operatorname{rank}}

\newcommand{\Spec}{\operatorname{Spec}}

\newcommand{\Aut}{\operatorname{Aut}}

\newcommand{\pp}{\mathbb{P}}

\newcommand{\qq}{\mathbb{Q}}
\newcommand{\zz}{\mathbb{Z}}

\newcommand{\rr}{\mathbb{R}}

\newcommand{\kk}{\mathbb{K}}

\def\O#1.{\mathcal {O}_{#1}}			
\def\pr #1.{\mathbb P^{#1}}				
\def\af #1.{\mathbb A^{#1}}			
\def\ses#1.#2.#3.{0\to #1\to #2\to #3 \to 0}	
\def\xrar#1.{\xrightarrow{#1}}			
\def\K#1.{K_{#1}}						
\def\bA#1.{\mathbf{A}_{#1}}			
\def\bM#1.{\mathbf{M}_{#1}}				
\def\bL#1.{\mathbf{L}_{#1}}				
\def\bB#1.{\mathbf{B}_{#1}}				
\def\bK#1.{\mathbf{K}_{#1}}			
\def\subs#1.{_{#1}}					
\def\sups#1.{^{#1}}

\usepackage{tikz}
\usetikzlibrary{matrix,arrows,decorations.pathmorphing}

\newtheorem{introdef}{Definition}

  \newtheorem{theorem}{Theorem}[section]
  \newtheorem{introthm}{Theorem}

  \newtheorem{introcor}{Corollary}

  \newtheorem{lemma}[theorem]{Lemma}
  \newtheorem{proposition}[theorem]{Proposition}
  \newtheorem{corollary}[theorem]{Corollary}

  \newtheorem{notation}[theorem]{Notation}
  \newtheorem{definition}[theorem]{Definition}
  \newtheorem{example}[theorem]{Example}

  \newtheorem{question}[theorem]{Question}

\newtheorem{remark}[theorem]{Remark}

\theoremstyle{remark}

\numberwithin{equation}{section}

\usepackage[all]{xy}

\begin{document}

\title[Log Calabi--Yau pairs of complexity zero and arbitrary index]{Log Calabi--Yau pairs of complexity zero and arbitrary index}

\author[J. Enwright]{Joshua Enwright}
\address{UCLA Mathematics Department, Box 951555, Los Angeles, CA 90095-1555, USA
}
\email{jlenwright1@math.ucla.edu}

\author[F.~Figueroa]{Fernando Figueroa}
\address{Department of Mathematics, Northwestern University, Evanston, Il 60208, USA
}
\email{fernando.figueroa@northwestern.edu}

\subjclass[2020]{Primary 14E05, 14J32;
Secondary 14E30, 14M25.}

\begin{abstract}
In this article, we give a characterization of log Calabi--Yau pairs of complexity zero and arbitrary index. As an application, we show that a log Calabi--Yau pair of birational complexity zero admits a crepant birational model which is a generalized Bott tower.

\end{abstract}

\maketitle

\setcounter{tocdepth}{1} 
\tableofcontents

\section{Introduction}
Throughout this article, we work over an algebraically closed field $\mathbb{K}$ of characteristic zero.\par
The \textit{complexity} of a log pair $(X,B)$ is defined as
$$c(X,B)=\dim X +\rank {\rm WDiv}_{\rm alg}(X)- |B|,$$
where ${\rm WDiv}_{\rm alg}(X)$ is the group of Weil divisors on $X$ modulo algebraic equivalence and $|B|$ is the sum of the coefficients of $B$. While not particularly well-behaved for arbitrary log pairs, the complexity enjoys remarkable properties when restricted to certain sub-classes of pairs. In ~\cite[Theorem 1.2, Corollary 1.3]{BMSZ18}, Brown, McKernan, Svaldi and Zong show the following:

\begin{introthm}[{\cite[Theorem 1.2, Corollary 1.3]{BMSZ18}}]\label{introthm:BMSZ18}
Let $(X,B)$ be a log canonical pair with $-(K_X+B)$ nef. Then $c(X,B)\geq 0.$ If $c(X,B)<1,$ then there is a toric log Calabi--Yau pair $(X,\Delta)$ with $\lfloor B \rfloor \leq \Delta$ and all but possibly $\lfloor 2c(X,B)\rfloor$ components of $\Delta$ in the support of $B.$
\end{introthm}

Our aim in this article is to completely characterize when the equality $c(X,B)=0$ holds in this lower bound on complexity. A log canonical pair $(X,B)$ with $-(K_X+B)$ nef and $c(X,B)=0$ is necessarily a \textit{log Calabi--Yau} pair (see Definition ~\ref{def:logCY} and Remark ~\ref{rmk:compl-zero-implies-logCY}). Therefore, in this article, we restrict to the class of log Calabi--Yau pairs. For log Calabi--Yau pairs of complexity zero, Theorem ~\ref{introthm:BMSZ18} takes on the following simpler form.


\begin{introcor}\label{introcor:BMSZ18}
Let $(X,B)$ be a log Calabi--Yau pair of complexity zero. Then there exists a toric log Calabi--Yau pair $(X,\Delta)$ satisfying
$$ \lfloor B\rfloor \leq \Delta \leq \lceil B\rceil.$$
\end{introcor}
There are two natural sources of examples of log Calabi--Yau pairs of complexity zero. The first of these comes from toric log Calabi--Yau pairs (see Definition ~\ref{def:toric-logCY-pair}). Indeed, given a toric log Calabi--Yau pair $(X,\Delta),$ both $\dim X+\rank {\rm WDiv}_{\rm alg}(X)$ and $|\Delta|$ can be identified with the number of rays in the fan of $X$. The second comes from the observation that if $(X,B_1),\hdots, (X,B_r)$ is a collection of log Calabi--Yau pairs of complexity zero supported on a fixed variety $X$ and if $b_1,\hdots, b_r\in [0,1]$ are rational numbers with sum $\sum_{i=1}^rb_i=1,$ then $\left(X,\sum_{i=1}^rb_iB_i\right)$ will be another log Calabi--Yau pair of complexity zero. Indeed, log canonicity, numerical triviality of the log canonical divisor, and the property that the coefficients in the boundary sum to $\dim X+\rank {\rm WDiv}_{\rm alg}(X)$ are all preserved under the formation of such convex combinations of boundary divisors. Our first main theorem states that these simple sources of examples of log Calabi--Yau pairs of complexity zero, in fact, produce all examples.

\begin{introthm}\label{introthm:convex-toric-div}
Let $(X,B)$ be a log Calabi--Yau pair of complexity zero. Then there are toric log Calabi--Yau pairs $(X,\Delta_1),\hdots, (X,\Delta_r)$ and rational numbers $b_1,\hdots, b_r\in [0,1]$ with sum $\sum_{i=1}^rb_i=1$ satisfying $$B=\sum_{i=1}^rb_i\Delta_i.$$
\end{introthm}
It follows easily from Corollary ~\ref{introcor:BMSZ18} that log Calabi--Yau pairs of index one\footnote{The {\em index} of a log Calabi--Yau pair $(X,B)$ is the smallest positive integer $m$ for which $m(K_X+B)\sim 0$.} and complexity zero are toric log Calabi--Yau pairs (cf. Lemma ~\ref{lem:compl-zero-ind-one}). Theorem ~\ref{introthm:convex-toric-div} can be seen as the analogue of this statement in the absence of restrictions on the index; log Calabi--Yau pairs of complexity zero are \textit{toric boundary arrangements} (see Definition ~\ref{def:toric-bound-arrang}). See Lemma ~\ref{lem:conjugate-tori} for an explanation of the relationship between different toric log Calabi--Yau pairs supported on a fixed variety.\par
A variant of the complexity which is natural from the perspective of birational geometry is the \textit{birational complexity}, which is defined for a log pair $(X,B)$ as $$c_{\rm bir}(X,B)=\inf\left\{c(Y,B_Y)\mid (X,B)\simeq _{\rm bir}(Y,B_Y)\right\}.$$ Here, the symbol $\simeq_{\rm bir}$ denotes crepant birational equivalence between log pairs (see Definition ~\ref{def:crepant}). The notion of birational complexity was introduced by Mauri and Moraga in ~\cite{MM24}, where they showed that every log Calabi--Yau $n$-fold of index one and birational complexity zero is crepant to $(\pp^n,H_0+\hdots+H_n)$ (see \cite[Theorem 1.6]{MM24}). Here, the $H_i$'s denote the coordinate hyperplanes in $\pp^n$. Our second main theorem offers an analogue to this result in the absence of restrictions on the index. In order to state it, we must first make the following definition.

\begin{introdef}\label{introdef:gen-bott-tower}{\em
We say that a projective variety $X$ is a \textit{generalized Bott tower} if there exists a tower of morphisms $$X=X_0\xrightarrow{f_0}X_1\rightarrow\hdots\rightarrow X_k\xrightarrow{f_k}X_{k+1}=\Spec\kk,$$
where each $f_i\colon X_{i}\rightarrow X_{i+1}$ is the projective space bundle associated to a direct sum of line bundles on $X_{i+1}$.
}
\end{introdef}

\begin{introthm}\label{introthm:gen-bott-towers}
Let $(X,B)$ be a log Calabi--Yau pair of birational complexity zero. Then there is a crepant birational model $(Y,B_Y)$ of $(X,B)$ satisfying:
\begin{enumerate}
    \item $(Y,B_Y)$ is a log Calabi--Yau pair of complexity zero,
    \item $Y$ is a generalized Bott tower.
\end{enumerate}
In particular, $Y$ is a smooth projective toric variety.
\end{introthm}
In contrast to Mauri and Moraga's result in the index one case, one is generally forced to consider infinitely many different generalized Bott towers $Y$ in each dimension when restrictions on the index are weakened (see Example ~\ref{ex:inf-many-models}).

\subsection{Geometry of log canonical centers}

One of the main challenges in proving Theorem ~\ref{introthm:gen-bott-towers} lies in understanding and controlling the collection of exceptional divisors with log discrepancy in the interval $[0,1],$ as these are precisely those divisors that can be extracted while maintaining an effective boundary. The complexity of a pair will remain unchanged after extracting only divisors with log discrepancy zero and will increase after extracting divisors with positive log discrepancy (see Lemma ~\ref{lem:complexity-and-exceptional-divisors}). One of the helpful features in the index one case is that any divisor with log discrepancy in $[0,1)$ automatically has log discrepancy zero. This offers a level of flexibility in performing complexity-preserving birational modifications which is simply not present in the absence of such strong restrictions on the index. To proceed, we will need new ways to find exceptional divisors with log discrepancy zero.\par
In order to state our first result in this direction, we begin by making the following definition.
\begin{introdef}{\em
We say that a Weil divisor $\Delta$ on $X$ is \textit{associated} to the log pair $(X,B)$ if the following conditions hold:
\begin{enumerate}
\item $(X,\Delta)$ is a toric log Calabi--Yau pair,
\item $\lfloor B \rfloor \leq \Delta \leq \lceil B \rceil.$
\end{enumerate}
We write 
\[
\mathcal{A}(X,B)=\left\{\Delta\in \text{WDiv}(X)\mid \Delta\text{ is associated to }(X,B)\right\}.
\]
}
\end{introdef}
The set $\mathcal{A}(X,B)$ is finite for any log pair $(X,B),$ and we can rephrase Corollary ~\ref{introcor:BMSZ18} by saying that $\mathcal{A}(X,B)$ is nonempty whenever $(X,B)$ is a log Calabi--Yau pair of complexity zero. With Theorem ~\ref{introthm:convex-toric-div} in hand, we will be able to prove the following:
\begin{introthm}\label{introthm:lccs}
Let $(X,B)$ be a log Calabi--Yau pair of complexity zero. Then the following statements hold:
\begin{enumerate}
    \item If a divisor $E$ over $X$ is toric with respect to every $\Delta\in \mathcal{A}(X,B),$ then $E$ is a log canonical place of $(X,B).$
    \item If a divisor $E$ over $X$ is a  log canonical place of $(X,B),$ then $E$ is toric with respect to some $\Delta \in \mathcal{A}(X,B).$
    \item If a subvariety $Z\subset X$ is a toric stratum of every $\Delta\in \mathcal{A}(X,B),$ then $Z$ is a log canonical center of $(X,B).$
    \item If a subvariety $Z\subset X$ is a  log canonical center of $(X,B),$ then $Z$ is a toric stratum of some $\Delta \in \mathcal{A}(X,B).$
\end{enumerate}
\end{introthm}

Given a toric log Calabi--Yau pair $(X,\Delta),$ its toric strata can be characterized either as the log canonical centers of the pair or as the closures of the orbits of the corresponding torus action on $X.$ Consequently, one obtains the well-known fact that a toric log Calabi--Yau pair $(X,\Delta)$ is log smooth if and only if it is divisorially log terminal.\footnote{Roughly, a
log pair $(X,B)$ is \textit{divisorially log terminal} if it is a log canonical pair which is log smooth generically along each of its log canonical centers. See Definition ~\ref{def:dlt} for a precise definition.} Using the information about log canonical places and centers provided by Theorem ~\ref{introthm:lccs}, we will be able to generalize this well-known fact to all log Calabi--Yau pairs of complexity zero via the following:

\begin{introthm}\label{introthm:dlt-log-smth}
Let $(X,B)$ be a log Calabi--Yau pair of complexity zero and let $(Y,B_Y)\rightarrow (X,B)$ be a dlt modification. Then the pair $(Y,\lfloor B_Y\rfloor)$ is log smooth. 
\end{introthm}

Theorem ~\ref{introthm:dlt-log-smth} will provide us with the complexity-preserving birational modifications that we need in order to prove Theorem ~\ref{introthm:gen-bott-towers}.

\subsection{Behavior with respect to contractions}

The property of being a toric variety imposes strong restrictions on the morphisms and rational maps that a variety can admit. Two cases which illustrate this quite clearly are the cases of contraction morphisms and birational contraction maps (see Section ~\ref{subsec:contractions}).\par
If $(X,\Delta)$ is a toric log Calabi--Yau pair and $f\colon X \dashrightarrow Y$ is a birational contraction map, then $(Y,\Delta_Y=f_*\Delta)$ is a toric log Calabi--Yau pair and the map $f$ is automatically toric (see Lemma ~\ref{lem:toric-bir-contr}). Moreover, the exceptional locus of such a map $f$ will be a union of toric strata. The following theorem generalizes these facts to all log Calabi--Yau pairs of complexity zero. 



\begin{introthm}\label{introthm:exceptional-loci}
Let $(X,B)$ be a log Calabi--Yau pair of complexity zero and let $f\colon X\dashrightarrow Y$ be a birational contraction. Then $(Y,B_Y=f_*B)$ is a log Calabi--Yau pair of complexity zero, and the exceptional locus of $f$ is a union of log canonical centers of the pair $(X,B).$ 
\end{introthm}

If $X$ is a toric variety and $f\colon X \rightarrow Y$ is a contraction, then $Y$ automatically admits the structure of a toric variety in such a way that $f$ becomes equivariant (see Lemma ~\ref{lem:toric-contractions}). Moreover, if $\Delta$ is the toric boundary of $X$ and $(Y, B_Y, \mathbf{M})$ is the generalized pair induced on $Y$ by the canonical bundle formula, then it follows that $B_Y=\Delta_Y$ is the toric boundary on $Y$ and $\mathbf{M}\sim 0$ (see Lemma ~\ref{lem:toric-canonical-bundle-formula}). The following theorem generalizes these facts to the case of log Calabi--Yau pairs of complexity zero.

\begin{introthm}\label{introthm:canonical-bundle-formula}
Let $(X,B)$ be a log Calabi--Yau pair of complexity zero, and let $B=\sum_{i=1}^rb_i\Delta_i$ be a decomposition as in Theorem ~\ref{introthm:convex-toric-div}. Let $X\rightarrow Y$ be a fibration, and let $(Y, B_Y,\mathbf{M})$ be the generalized pair induced by the canonical bundle formula. For each $1\leq i \leq r,$ denote by $(Y,\Gamma_i)$ the toric log Calabi--Yau pair induced by $(X,\Delta_i).$ Then $B_Y=\sum_{i=1}^rb_i\Gamma_i.$ In particular, $(Y,B_Y)$ is a log Calabi--Yau pair of complexity zero and $\mathbf{M}\sim_\qq 0$ where it descends. 
\end{introthm}

\subsection*{Acknowledgements}

The authors would like to thank Stefano Filipazzi, Lena Ji, Joaqu\'in Moraga and Burt Totaro for very useful comments. The first author was partially supported by NSF grant DMS-2054553.

\section{Preliminaries}

In this section we collect some preliminaries about log pairs, the relative MMP, toric geometry and the complexity of pairs.

\subsection{Log pairs} In this subsection we recall some basic definitions regarding log pairs and log discrepancies.

\begin{definition}{\em
A \textit{log sub-pair} $(X,B)$ consists of a normal quasi-projective variety $X$ and a $\qq$-divisor $B$ with the property that $K_X+B$ is $\qq$-Cartier. We say that a log sub-pair $(X,B)$ is a \textit{log pair} if $B$ is effective.}
\end{definition}

\begin{definition}{\em
Let $(X,B)$ be a log sub-pair, and let $f\colon Y\rightarrow X$ be a proper birational morphism from a normal variety $Y$. We will refer to the unique log sub-pair $(Y,B_Y)$ satisfying
\begin{itemize}
    \item $K_Y+B_Y\sim_\qq f^*(K_X+B),$ and
    \item $f_*B_Y=B$
\end{itemize}
as the \textit{log pullback of $(X,B)$ via $f.$}}
\end{definition}

\begin{definition}{\em
Let $(X,B)$ be a log sub-pair and let $f\colon Y \rightarrow X$ be a projective birational morphism from a normal variety $Y.$ Given a prime divisor $E\subset Y,$ its \textit{log discrepancy with respect to $(X,B)$} is the quantity
$$a_E(X,B)=1-{\rm coeff}_E(B_Y),$$
where $(Y,B_Y)$ is the log pullback of $(X,B)$ via $f.$\par We say that a log pair $(X,B)$ is \textit{log canonical} (respectively \textit{Kawamata log terminal}) if $a_E(X,B)\geq 0$ (respectively $a_E(X,B)> 0$) for all prime divisors $E$ over $X.$}
\end{definition}

\begin{definition}{\em
Let $(X,B)$ be a log canonical pair. A \textit{log canonical place} of $(X,B)$ is a divisor $E$ over $X$ for which $a_E(X,B)=0$. A \textit{log canonical center} of $(X,B)$ is a subvariety $Z\subset X$ which is the image on $X$ of a log canonical place of $(X,B)$.}
\end{definition}

\begin{definition}\label{def:dlt}
{\em
Let $(X,B)$ be a log pair. We say that $(X,B)$ is \textit{divisorially log terminal} if it is log canonical and there exists an open subset $U \subset X$, such that:
\begin{enumerate}
\item the pair $(U,B_U)$ is log smooth, and
\item every log canonical center of $(X,B)$ intersects $U.$
\end{enumerate}
}
\end{definition}

\begin{notation}{\em
We will often abbreviate log canonical, Kawamata log terminal and divisorially log terminal as lc, klt and dlt respectively.}
\end{notation}

\begin{definition}\label{def:dlt-mod}
{\em Let $(X,B)$ be a log canonical pair.  A \textit{dlt modification} of $(X,B)$ is a projective birational morphism $f \colon Y \rightarrow X$, satisfying the following:

\begin{enumerate}
\item $Y$ is $\qq$-factorial,
\item every $f$-exceptional divisor is a log canonical place of $(X,B)$, 
\item the log pullback $(Y,B_Y)$ of $(X,B)$ via $f$ is dlt, and
\item $K_Y+B_Y$ is $f$-nef.
\end{enumerate}
}
\end{definition}

We recall the existence of dlt modifications.

\begin{lemma}\label{lem:exist-dlt-mod}
Let $(X,B)$ be a log canonical pair. Then there exists a dlt modification $f\colon Y\rightarrow X$ of $(X,B)$.
\end{lemma}

\begin{proof}
 This follows from ~\cite[Corollary 1.36]{K13}.
\end{proof}


\begin{definition}\label{def:logCY}
{\em Let $(X,B)$ be a log pair. We say that $(X,B)$ is a \textit{log Calabi--Yau pair} if $X$ is projective, $(X,B)$ is log canonical and $K_X+B\equiv 0$.
}
\end{definition}

\begin{remark}\label{rmk:alt-logCY-def}
{\em
It follows from ~\cite[Theorem 1.5]{FG14a} that $(X,B)$ is a log Calabi--Yau pair in the sense of Definition ~\ref{def:logCY} if and only if it is a log canonical pair satisfying $K_X+B \sim_{\mathbb{Q}}0$.
}
\end{remark}

\begin{definition}\label{def:crepant}
    {\em
    Let $(X,B)$ and $(Y,B_Y)$ be two log sub-pairs. We will say that a birational map $f\colon X \dashrightarrow Y$ is \textit{crepant} with respect to these sub-pairs if it admits a resolution 
    \[
\xymatrix{ 
 & Z \ar[ld]_-{p} \ar[rd]^-{q} & \\ 
 X \ar@{-->}[rr]^-{f}& & Y
}
\]
with proper birational morphisms $p$ and $q$ such that the log pullback of $(X,B)$ via $p$ is equal to the log pullback of $(Y,B_Y)$ via $q.$
    }
\end{definition}

\begin{remark}\label{rmk:crepant-discrepancy}
{\em
Let $f\colon (X,B)\dashrightarrow (Y,B_Y)$ be a crepant birational map between two log pairs. It follows from ~\cite[Lemma 2.30]{KM98} that $a_E(X,B)=a_E(Y,B_Y)$ for every divisor $E$ over $X$ and $Y.$
}
\end{remark}

\subsection{Contractions}\label{subsec:contractions} In this subsection we recall the definition of contractions and fibrations.

\begin{definition}{\em
    Let $f\colon X \rightarrow Y$ be a proper morphism of varieties. We denote by $$N^1(X/Y)=\left({\rm Pic}(X)\otimes_\zz\rr\right)/\equiv_Y$$
    and $$N_1(X/Y)=\left(Z_1(X/Y)\otimes_\zz\rr\right)/\equiv_Y$$
    the real vector spaces of Cartier divisors and relative $1$-cycles modulo numerical equivalence over $Y,$ respectively (see ~\cite[Section IV.4]{Kle66}). These vector spaces are dual under the intersection pairing, and are finite-dimensional by ~\cite[Proposition IV.4.3]{Kle66}. Their common dimension is denoted by $\rho(X/Y)$ and is referred to as the \textit{relative Picard rank} of the morphism $f.$
    }
\end{definition}

\begin{remark}{\em
When $Y=\Spec \kk$ is a point, we will simply write $N^1(X), N_1(X)$ and $\rho(X)$ and will omit the word ``relative''.
}
\end{remark}

\begin{definition}\label{def:contraction}{\em
We say that a morphism $f\colon X \rightarrow Y$ between normal quasi-projective varieties is a \textit{contraction} if it is projective and satisfies $f_*\mathcal{O}_X=\mathcal{O}_Y$. We say that the contraction $f$ is a \textit{fibration} if $\dim Y < \dim X.$ We say that the contraction $f$ is \textit{extremal} if $\rho(X/Y)=1.$
}
\end{definition}

\begin{remark}{\em
A contraction $f\colon X \rightarrow Y$ is birational if and only if $\dim Y = \dim X,$ and a projective birational morphism between normal varieties is automatically a contraction.
}
\end{remark}

\begin{definition}\label{def:bir-contraction-map}{\em
We say that a birational map $f\colon X \dashrightarrow Y$ between normal quasi-projective varieties is a \textit{birational contraction map} if it is surjective in codimension one.
}
\end{definition}

\subsection{Mori Dream Spaces}
In this subsection we recall the definition and basic properties of \textit{Mori dream spaces}, as introduced by Hu and Keel in \cite{HK00}.
\begin{definition}\label{def:sqm}{\em
    Let $X$ be a normal projective variety. We say that a birational map $f\colon X \dashrightarrow Y$ is a \textit{small $\qq$-factorial modification} of $X$ if $f$ is an isomorphism in codimension one and $Y$ is a normal, $\qq$-factorial projective variety.
    }
\end{definition}

\begin{definition}\label{def:mds}{\em
We say that a normal projective variety $X$ is a \textit{Mori dream space} if
\begin{enumerate}
    \item $X$ is $\qq$-factorial and ${\rm Pic}(X)_\rr=N^1(X),$
    \item ${\rm Nef}(X)$ is the affine hull of finitely many semi-ample line bundles,
    \item there is a finite collection of small $\qq$-factorial modifications $f_i\colon X \dashrightarrow X_i$ such that each $X_i$ satisfies (1) and (2) and such that ${\rm Mov}(X)$ is the union of the $f_i^*\left({\rm Nef}(X_i)\right).$
\end{enumerate}
}
\end{definition}

The following propositions summarize the properties of Mori dream spaces that will be most important for us. They follow from ~\cite[Proposition 1.11]{HK00}.

\begin{proposition}\label{prop:mds-mmp}
Let $X$ be a Mori dream space and let $f\colon X \rightarrow Y$ be a projective morphism. Then the $D$-MMP over $Y$ can be run for any divisor $D$ on $X$, in the sense that all necessary contractions and flips exist. Any such MMP terminates. 
\end{proposition}

\begin{proposition}\label{prop:mds-nef}
Let $X$ be a Mori dream space and let $D$ be a nef divisor on $X.$ Then $D$ is semi-ample.
\end{proposition}

\subsection{The relative MMP}
In this subsection we prove a few lemmas about the relative MMP.

\begin{lemma}\label{lem:rel-cone-curves-over-open}
Let $f\colon X\rightarrow Y$ be a projective morphism between normal varieties with $X$ $\qq$-factorial, and let $U\subset Y$ be an open subset. Then the assignment $[C]\mapsto [C]$ defines an injective linear map $N_1\left(f^{-1}(U)/U\right)\hookrightarrow N_1(X/Y)$ mapping $\overline{NE}(f^{-1}(U)/U)$ into $\overline{NE}(X/Y).$
\end{lemma}
\begin{proof}
    Well-definedness follows from the fact that $L\cdot \left(\sum_{i=1}^na_i[C_i]\right)=L|_{f^{-1}(U)}\cdot \left(\sum_{i=1}^na_i[C_i]\right)$ for all $L\in {\rm Pic}(X)$ and all relative curves $C_1,\hdots, C_n$ over $U.$ For injectivity, we use the $\qq$ factoriality of $X$ to conclude that the restriction homomorphism ${\rm Pic}(X)_{\qq}\rightarrow {\rm Pic}\left(f^{-1}(U)\right)_\qq$ is surjective. 
\end{proof}

\begin{lemma}\label{lem:rel-mmp-over-open-single-step}
    Let $f\colon X\rightarrow Y$ be a projective morphism between normal varieties with $X$ $\qq$-factorial, and let $U\subset Y$ be an open subset. Let $E$ be a $\qq$-divisor on $X.$ Let 
\[
\xymatrix{  
 X \ar[rd]_-{f} \ar@{-->}[rr]^-{\phi}& & W \ar[ld]^-{g}\\
 & Y &
}
\]
be a single step of an $E$-MMP over $Y.$ Then, either $\phi|_{f^{-1}(U)}$ induces an isomorphism $f^{-1}(U)\xrightarrow{\cong}g^{-1}(U)$ or 
\[
\xymatrix{  
 f^{-1}(U) \ar[rd]_-{f|_{f^{-1}(U)}} \ar@{-->}[rr]^-{\phi|_{f^{-1}(U)}}& & g^{-1}(U) \ar[ld]^-{g|_{g^{-1}(U)}}\\
 & U &
}
\]
is a single step of an $E|_{f^{-1}(U)}$-MMP over $U.$
\end{lemma}
\begin{proof}
    First, we consider the case in which $\phi$ is a morphism. Thus, $\phi$ is an $E$-negative extremal contraction over $Y.$ It follows from flat base change ~\cite[Proposition 9.3]{Har77} that $\phi|_{f^{-1}(U)}\colon f^{-1}(U)\rightarrow g^{-1}(U)$ is a contraction. If all curves contracted by $\phi$ are contained in $X\setminus f^{-1}(U),$ then $\phi|_{f^{-1}(U)}$ is an isomorphism. Otherwise, it is an $E|_{f^{-1}(U)}$-negative extremal contraction over $U.$\par
    From now on, assume that $\phi$ is not a morphism. Thus, $\phi$ is an isomorphism in codimension one and there is a commutative diagram 
\[
\xymatrix{  
 X \ar[rd]_-{\psi} \ar@/_1pc/[rdd]_-{f} \ar@{-->}[rr]^-{\phi}& & W \ar[ld]^-{\psi^+} \ar@/^1pc/[ldd]^-{g} \\
 & Z \ar[d]_-{h} & \\
 & Y & 
}
\]
in which 
\begin{enumerate}
    \item $\psi$ is a small $E$-negative extremal contraction, 
    \item $\psi^+$ is a small $\phi_*E$-positive extremal contraction. 
\end{enumerate}
It follows from flat base change ~\cite[Proposition 9.3]{Har77} that $\psi|_{f^{-1}(U)}$ is a contraction. If all curves contracted by $\psi$ are contained in $X\setminus f^{-1}(U),$ then $\psi|_{f^{-1}(U)}$ induces an isomorphism $f^{-1}(U)\xrightarrow{\cong}h^{-1}(U)$ and $\phi|_{f^{-1}(U)}$ induces an isomorphism $f^{-1}(U)\xrightarrow{\cong}g^{-1}(U).$ Otherwise, $\psi|_{f^{-1}(U)}$ is a small $E|_{f^{-1}(U)}$-negative extremal contraction and $\psi^+|_{g^{-1}(U)}$ is a small extremal contraction which is positive with respect to $(\phi_*E)|_{g^{-1}(U)}=(\phi|_{f^{-1}(U)})_*(E|_{f^{-1}(U)}).$ It follows, in this case, that $\psi^+|_{g^{-1}(U)}$ is the flip of $\psi|_{f^{-1}(U)}.$
\end{proof}

\begin{lemma}\label{lem:rel-mmp-over-open}
    Let $f\colon X\rightarrow Y$ be a projective morphism between normal varieties with $X$ $\qq$-factorial, and let $U\subset Y$ be an open subset. Let $E$ be a $\qq$-divisor on $X.$ Let 
\[
\xymatrix{ 
X = X_0 \ar@{-->}[r]^-{g_0}
 & X_1 \ar@{-->}[r]^-{g_1}
 & \hdots \ar@{-->}[r]^-{g_{m-1}}
 & X_m
}
\]
be a sequence of steps of an $E$-MMP over $Y.$ For each $0\leq i \leq m-1,$ denote by $V_i\subset X_i$ the preimage of $U$ in $X_i$ and by $h_i\colon V_i \dashrightarrow V_{i+1}$ the restriction of $g_i.$ Then there are indices $0\leq i_0<\hdots<i_{r-1}\leq m-1$ such that $h_i$ is an isomorphism for $i\notin \{i_0,\hdots, i_{r-1}\}$ and
\[
\xymatrix{ 
V_0= V_{i_0} \ar@{-->}[r]^-{h_{i_0}}
 & V_{i_1} \ar@{-->}[r]^-{h_{i_1}}
 & \hdots \ar@{-->}[r]^-{h_{i_{r-1}}}
 & V_{i_r}
}
\]
is a sequence of steps of an $E|_{V_0}$-MMP over $U.$
\end{lemma}

\begin{proof}
    This follows from repeated application of Lemma ~\ref{lem:rel-mmp-over-open-single-step}.
\end{proof}

\begin{lemma}\label{lem:rel-cone-curves-product}
    Let $X$ and $Y$ be normal varieties with $X$ projective and $\qq$-factorial. The assignment $[C]\mapsto [C\times\{y\}]$ is independent of the choice of $y\in Y$ and defines isomorphisms $N_1\left(X\right)\xrightarrow{\cong} N_1(X\times Y/Y)$ and $\overline{\rm NE}\left(X\right)\xrightarrow{\cong}\overline{\rm NE}\left(X\times Y/Y\right).$
\end{lemma}
\begin{proof}
    Well-definedness follows from the fact that $L\cdot \left(\sum_{i=1}^na_i[C_i\times\{y\}]\right)=L_{y}\cdot \left(\sum_{i=1}^na_i[C_i]\right)$ for all $L\in {\rm Pic}(X\times Y)$ and all curves $C_1,\hdots, C_n \subset X.$ For injectivity, we use the fact that $L\cong (pr_1^*L)_y$ for all $L\in {\rm Pic}(X).$ To see that the homomorphism is independent of choice of $y\in Y,$ consider a curve $C\subset X.$ The family $C\times Y$ is flat over $Y,$ from which it follows that $L_{y_1}\cdot C=L_{y_2}\cdot C$ for all $L\in {\rm Pic}(X\times Y)$ and $y_1,y_2\in Y.$ Surjectivity follows from independence from choice of $y\in Y,$ since every class in $N^1(X\times Y/Y)$ is represented by some curve of the form $C\times \{y'\}$ for $C\subset X$ and $y'\in Y.$
\end{proof}

\begin{lemma}\label{lem:fiberwise-mmp-single-step}
Let $X$ and $Y$ be normal varieties with $X$ a Mori dream space. Let $E_X$ be a $\qq$-divisor on $X$ and write $E=pr_1^*E_X,$ where $pr_1\colon X\times Y \rightarrow X$ is the projection onto $X.$ Let \[
\xymatrix{ 
X\times Y \ar@{-->}[r]^-{\widetilde{\phi}}
 & \widetilde{W} 
}
\]
be a single step of an $E$-MMP over $Y.$ Then there is a single step 
\[
\xymatrix{ 
X \ar@{-->}[r]^-{\phi}
 & W
}
\]
of an $E_X$-MMP satisfying 
\begin{enumerate}
    \item $\widetilde{W}\cong W\times Y$, and
    \item the identification in {\rm (1)} identifies $\widetilde{\phi}$ with $\phi\times Y$.
\end{enumerate}
\end{lemma}
\begin{proof}
First, we consider the case that $\widetilde{\phi}$ is a morphism. Thus, $\widetilde{\phi}$ is an $E$-negative extremal contraction over $Y.$ Choose a closed point $y\in Y$ and let $$X\xrightarrow{\phi}W\rightarrow \widetilde{W}$$ be the Stein factorization of $$X \xhookrightarrow{i_y}X\times Y \xrightarrow{\widetilde{\phi}}\widetilde{W}.$$ Then $\phi$ is an $E_X$-negative extremal contraction, and it follows from flat base change ~\cite[Proposition 9.3]{Har77} that $\phi\times Y \colon X\times Y \rightarrow W\times Y$ is a contraction as well. The contractions $\phi\times Y$ and $\widetilde{\phi}$ contract exactly the same curves by Lemma ~\ref{lem:rel-cone-curves-product}, so it follows from the rigidity lemma ~\cite[Lemma 1.15]{Deb01} that there is an isomorphism $\psi\colon W\times Y\rightarrow \widetilde{W}$ making the diagram \[
\xymatrix{ 
 & X\times Y \ar[ld]_-{\phi\times Y} \ar[rd]^-{\widetilde{\phi}} & \\ 
 W\times Y \ar[rr]^-{\psi}& & \widetilde{W}
}
\]
commutative. The desired result now follows in this case.\par
From now on, assume that $\widetilde{\phi}$ is not a morphism. Thus, $\widetilde{\phi}$ is an isomorphism in codimension one and there is a commutative diagram 
\[
\xymatrix{  
 X\times Y \ar[rd]_-{\widetilde{\psi}} \ar@{-->}[rr]^-{\widetilde{\phi}}& & \widetilde{W} \ar[ld]^-{\widetilde{\psi}^+}\\
 & \widetilde{Z} &
}
\]
in which 
\begin{enumerate}
    \item $\widetilde{\psi}$ is a small $E$-negative extremal contraction, 
    \item $\widetilde{\psi}^+$ is a small $\widetilde{\phi}_*E$-positive extremal contraction.
\end{enumerate}
By the arguments of the previous paragraph, we may identify the morphism $\widetilde{\psi}$ with a morphism of the form $$\psi\times Y \colon X\times Y \rightarrow Z \times Y$$
for some $E_X$-negative extremal contraction $\psi\colon X \rightarrow Z.$ Note that $\psi$ must be a small birational contraction since $\widetilde{\psi}$ is. Let $\psi^+\colon W \rightarrow Z$ be the flip of $\psi.$ Denote by $\phi\colon X \dashrightarrow W$ the induced birational map, which is an isomorphism in codimension one, and by $E_W=\phi_*E_X$ the strict transform of $E_X$ on $W.$ We obtain a commutative diagram 
\[
\xymatrix{  
 X\times Y \ar[rd]_-{\psi\times Y} \ar@{-->}[rr]^-{\phi\times Y}& & W\times Y \ar[ld]^-{\psi^+\times Y}\\
 & Z\times Y &
}
\]
in which 
\begin{enumerate}
    \item $\psi\times Y$ is a small $E$-negative extremal contraction, 
    \item $\psi^+\times Y$ is a small $pr_1^*E_W$-positive extremal contraction.
\end{enumerate}
Using the fact that $(\phi\times Y)_*E=pr_1^*E_W$, it follows from ~\cite[Lemma 6.2]{KM98} that we may identify $\widetilde{\psi}^+$ with $\psi^+\times Y,$ hence also $\widetilde{\phi}$ with $\phi\times Y.$
\end{proof}

\begin{lemma}\label{lem:fiberwise-mmp}
    Let $X$ and $Y$ be normal varieties with $X$ a Mori dream space. Let $E_X$ be a $\qq$-divisor on $X$ and write $E=pr_1^*E_X,$ where $pr_1\colon X\times Y \rightarrow X$ is the projection onto $X.$ Let 
\[
\xymatrix{ 
X\times Y = \widetilde{X}_0 \ar@{-->}[r]^-{\widetilde{f}_0}
 & \widetilde{X}_1 \ar@{-->}[r]^-{\widetilde{f}_1}
 & \hdots \ar@{-->}[r]^-{\widetilde{f}_{m-1}}
 & \widetilde{X}_m
}
\]
be a sequence of steps of an $E$-MMP over $Y.$ Then there is a sequence
\[
\xymatrix{ 
X = X_0 \ar@{-->}[r]^-{f_0}
 & X_1 \ar@{-->}[r]^-{f_1}
 & \hdots \ar@{-->}[r]^-{f_{m-1}}
 & X_m
}
\]
of steps of an $E_X$-MMP satisfying
\begin{enumerate}
    \item $\widetilde{X}_i\cong X_i\times Y$ for each $0\leq i \leq m,$
    \item the identifications in {\rm (1)} identify $\widetilde{f}_i$ with $f_i\times Y$ for each $0\leq i \leq m-1.$
\end{enumerate}
\end{lemma}

\begin{proof}
    This follows from repeated application of Lemma ~\ref{lem:fiberwise-mmp-single-step}.
\end{proof}

\begin{lemma}\label{lem:horiz-pseudoeff}
    Let $f\colon X\rightarrow Y$ be a projective morphism between normal varieties with $X$ $\qq$-factorial. Let $D$ be an effective divisor on $X$ such that $f({\rm Supp}(D))=Y.$ Then there exists a curve $C\subset X$ contracted to a point by $f$ which satisfies $D\cdot C>0.$
\end{lemma}
\begin{proof}
    It suffices to show the result for some positive multiple of $D.$ Thus, we may assume that $D$ is an effective Cartier divisor. Our assumptions imply that there is an irreducible component $F$ of some fiber of $f$ satisfying $F\nsubset {\rm Supp}(D).$ The effective Cartier divisor $D$ restricts to an effective Cartier divisor $D_F$ on $F.$ It follows from the projectivity of $f$ that $F$ is a projective variety, and so we may choose a curve $C\subset F$ which is the complete intersection of very ample divisors on $F$. Such a curve $C$ is contracted by $f$ and satisfies
        $$D\cdot C=D_F\cdot C>0.$$
\end{proof}

\subsection{Toric geometry}

We refer the reader to ~\cite{CLS11} for background on toric geometry. The following definition will be convenient for us.
\begin{definition}\label{def:toric-logCY-pair}
{\em
We say that an $n$-dimensional log pair $(X,\Delta)$ is a \textit{toric log Calabi--Yau pair} if it is log Calabi--Yau pair with $\Delta$ reduced and ${\rm Aut}^0(X,\Delta)$ is an $n$-dimensional algebraic torus.
}
\end{definition}

\begin{remark}\label{rmk:toric-def-property-vs-structure}
    {\em
    The notion of toric log Calabi--Yau pair defined above is a property of a log pair, and does not require the specification of any additional structure such as a group action. Given a toric log Calabi--Yau pair $(X,\Delta),$ however, the following lemma explains a canonical way to equip $X$ with the structure of a toric variety in the usual sense.
    }
\end{remark}

\begin{lemma}\label{lem:toric-defn-comparison}
Let $(X,\Delta)$ be an $n$-dimensional log Calabi--Yau pair with reduced boundary $\Delta.$ Then $(X,\Delta)$ is a toric log Calabi--Yau pair if and only if $X$ admits an algebraic action by an $n$-dimensional algebraic torus $\mathbb{T}$ for which $X\setminus \Delta$ is an orbit with trivial isotropy group. In this latter case, the action of $\mathbb{T}$ on $X$ induces an isomorphism ${\rm Aut}^0(X,\Delta)\cong\mathbb{T}.$
\end{lemma}
\begin{proof}
Suppose first that $X$ admits an action of an $n$-dimensional algebraic torus $\mathbb{T}$ for which $X\setminus \Delta$ is an orbit with trivial isotropy group. It follows that $\mathbb{T}$ is isomorphic to an algebraic subgroup of ${\rm Aut}^0(X,\Delta).$ To see that ${\rm Aut}^0(X,B)\cong\mathbb{T}$, it suffices to note that ${\rm dim}{\rm Aut}^0(X,\Delta)\leq n$ by \cite[Lemmas 2.1, 2.2]{Hu18}.\par
Conversely, suppose that $(X,\Delta)$ is toric log Calabi--Yau. In particular, ${\rm Aut}^0(X,\Delta)$ is an $n$-dimensional algebraic torus acting faithfully on $X.$ By \cite[Lemma 2.2]{Hu18}, there is a Zariski dense open orbit $U\subset X$ with trivial isotropy group. Thus, $X$ is a toric variety in the sense of \cite[Definition 3.1.1]{CLS11}. Since $\Delta$ is a proper closed subset of $X$ which is invariant under the action of ${\rm Aut}^0(X,\Delta),$ we must have $U\subset X\setminus \Delta.$ We claim that $U=X\setminus \Delta.$ It follows from \cite[Theorem 3.2.6]{CLS11} that $X\setminus U$ has pure codimension $1$ in $X;$ denote by $\Gamma$ the divisor which is the reduced sum of the components of this closed subset. To show that $U=X\setminus \Delta,$ it suffices to show that $\Gamma=\Delta$. We certainly have $\Delta \leq \Gamma.$ We have $\Delta\sim_{\qq}-K_X$ since $(X,\Delta)$ is log Calabi--Yau, and we have $\Gamma \sim -K_X$ by \cite[Theorem 8.2.3]{CLS11}. Thus, $\Gamma-\Delta$ is an effective divisor which is $\qq$-linearly trivial. Since $X$ is projective, it follows that $\Gamma-\Delta=0.$ 
\end{proof}

An important and well-known property of toric varieties is that they are Mori dreams spaces.
\begin{proposition}\label{prop:toric-mds}
    Let $X$ be a normal, $\qq$-factorial projective variety admitting the structure of a toric variety. Then $X$ is a Mori dream space.
\end{proposition}
\begin{proof}
    This follows from ~\cite[Theorem 2.1]{cox2014erratum} and ~\cite[Corollary 2.4]{HK00}.
\end{proof}

Next, we detail some special properties enjoyed by maps out of toric varieties. 

\begin{lemma}\label{lem:toric-bir-contr}
    Let $(X,\Delta)$ be a toric log Calabi--Yau pair and let $f\colon X \dashrightarrow Y$ be a birational contraction map to a projective variety $Y.$ Then $(Y,\Delta_Y=f_*\Delta)$ is a toric log Calabi--Yau pair, $f$ is toric and the exceptional locus of $f$ is a union of toric strata of $(X,\Delta).$
\end{lemma}
\begin{proof}
    That $(Y,\Delta_Y)$ and $f$ are toric follows from ~\cite[Lemma 2.3.2]{BMSZ18}. To show that the exceptional locus of $f$ is a union of toric strata, it suffices to show that its complement is an open torus-invariant subset. The complement of the exceptional locus is the subset of the domain of $f$ at which $f$ is an isomorphism, and torus-invariance follows from the fact that $f$ is toric.
\end{proof}

\begin{lemma}\label{lem:toric-contractions}
Let $X$ be a toric variety and let $f\colon X \rightarrow Y$ be a contraction. Then $Y$ admits the structure of a toric variety in such a way that $f$ becomes equivariant.
\end{lemma}
\begin{proof}
This is proven in ~\cite[Proposition 2.7]{tan23}.
\end{proof}

The following is well-known, but we provide a proof here for convenience.
\begin{lemma}\label{lem:splitting-fans}
    Let $f\colon X \rightarrow Y$ be an extremal fibration between $\qq$-factorial projective toric varieties. Then the fan $\Sigma_X$ of $X$ can be expressed as a sum $$\Sigma_X=\Sigma_{X,F}+\Sigma_{X,Y}$$
    of subfans $\Sigma_{X,F},\Sigma_{X,Y}\subset \Sigma_X$, where 
    \begin{enumerate}
        \item $\Sigma_{X,F}$ has support equal to ${\rm Ker}(f_*),$
        \item $f_*$ restricts to a bijection $\tau\xrightarrow{\cong}f_*(\tau)$ for each $\tau\in \Sigma_{X,Y},$ 
        \item the assignment $\tau\mapsto f_*(\tau)$ determines a bijection $\Sigma_{X,Y}\xrightarrow{\cong}\Sigma_Y.$
    \end{enumerate}
\end{lemma}
\begin{proof}
    As the morphism $f$ is proper, it follows that $f_*^{-1}(\sigma)$ is a union of cones in $\Sigma_X$ for each $\sigma\in \Sigma_Y.$ In particular, ${\rm Ker}(f_*)=f_*^{-1}(\{0\})$ is a union of cones in $\Sigma_X.$ Denoting by $\Sigma_{X,F}$ the collection of all cones in $\Sigma_X$ which are contained in ${\rm Ker}(f_*),$ we obtain a subfan of $\Sigma_X$ satisfying (1).\par
    We turn to define the fan $\Sigma_{X,Y}.$ We begin by noting that, for each $\sigma\in \Sigma_Y^{(1)},$ there is a unique $\widetilde{\sigma}\in \Sigma_X^{(1)}$ satisfying $f_*(\widetilde{\sigma})=\sigma.$ To see this, recall that, for each $\sigma\in \Sigma_Y^{(1)},$ $f_*^{-1}(\sigma)$ is a union of cones in $\Sigma_X.$ It follows that, for each $\sigma\in \Sigma_Y^{(1)},$ there is at least one $\widetilde{\sigma}\in \Sigma_{X}^{(1)}$ with $f_*(\widetilde{\sigma})=\sigma.$ Such a $\widetilde{\sigma}$ must necessarily be contained in $\Sigma_{X}^{(1)}\setminus\Sigma_{X,F}^{(1)}$. But, writing $r=\dim X- \dim Y=\dim {\rm Ker}(f_*)$, we have
    \begin{align*}
    \lvert\Sigma_{X}^{(1)}\setminus\Sigma_{X,F}^{(1)}\rvert&\leq(\dim X+\rho(X))-(r+1)\\
        &=\dim Y + \rho(Y)\\
        &= \lvert\Sigma_{Y}^{(1)}\rvert.
    \end{align*}
    The desired uniqueness follows.\par
    Using this notation, we define a fan $$\Sigma_{X,Y}=\{{\rm Cone}(\widetilde{\sigma_1},\hdots, \widetilde{\sigma_r})\vert\sigma_1,\hdots, \sigma_r\in \Sigma_Y^{(1)},{\rm Cone}(\sigma_1,\hdots, \sigma_r)\in \Sigma_Y\}.$$ To see that this is a subfan of $\Sigma_X,$ suppose $\widetilde{\sigma}\in \Sigma_{X,Y}.$ It follows from the definition of $\Sigma_{X,Y}$ that $\sigma=f_*(\widetilde{\sigma})\in \Sigma_Y.$ Write $\sigma_1,\hdots, \sigma_k\in \Sigma_Y^{(1)}$ for the rays spanning $\sigma.$ Since $f_*^{-1}(\sigma)$ is a union of cones in $\Sigma_X,$ there is some $\tau\in \Sigma_X$ with $f_*(\tau)=\sigma.$ This cone $\tau$ must contain the rays $\widetilde{\sigma}_1,\hdots,\widetilde{\sigma}_k.$ It follows that $\tau$ must contain $\widetilde{\sigma}={\rm Cone}(\widetilde{\sigma}_1,\hdots,\widetilde{\sigma}_k)$ as a face, hence that $\widetilde{\sigma}$ is a cone in $\Sigma_X.$ It is clear from the definitions that the fan $\Sigma_{X,Y}$ satisfies (3). That it satisfies (2) follows from the fact that the fans $\Sigma_X$ and $\Sigma_Y$ are simplicial together with the fact that, for each $\sigma\in \Sigma_Y,$ $\sigma$ and $\widetilde{\sigma}$ are generated by the same number of rays.\par
    Finally, we verify that $\Sigma_X=\Sigma_{X,F}+\Sigma_{X,Y}.$ That every cone in $\Sigma_X$ can be expressed as the sum of a cone in $\Sigma_{X,F}$ and a cone in $\Sigma_{X,Y}$ follows from the fact that every ray in $\Sigma_X^{(1)}$ is contained in either $\Sigma_{X,F}^{(1)}$ or $\Sigma_{X,Y}^{(1)}.$ It remains to show that $\tau+\widetilde{\sigma}\in \Sigma_X$ whenever $\tau\in \Sigma_{X,F}$ and $\widetilde{\sigma}\in \Sigma_{X,Y}.$ Choose $v\in {\rm Relint}(\tau)$ and $w\in {\rm Relint}(\widetilde{\sigma}).$ Writing $\sigma=f_*(\widetilde{\sigma}),$ we have $v+w\in f_*^{-1}(\sigma).$ Thus, there is a cone $\gamma\in \Sigma_X$ with $v+w\in \gamma \subset f_*^{-1}(\sigma).$ Since $f_*(w)\in f_*(\gamma)\cap{\rm Relint}(\sigma),$ we must have $\gamma=\tau'+\widetilde{\sigma}$ for some $\tau'\in \Sigma_{X,F}.$ To show that $\tau'=\tau,$ it suffices to show that $v\in \tau'.$ Since $v+w\in \tau'+\widetilde{\sigma},$ we can write $v+w=v'+w'$ for some $v'\in \tau'$ and $w'\in \widetilde{\sigma}.$ But $$f_*(w)=f_*(v+w)=f_*(v'+w')=f_*(w')$$ implies that $w=w',$ hence that $v=v'\in \tau'.$
\end{proof}

\begin{lemma}\label{lem:proj-bundle-sum-line-bundles}
Let $f\colon X \rightarrow Y$ be a toric morphism between $\qq$-factorial projective toric varieties. Assume that $f$ is a locally trivial fiber bundle with fiber isomorphic to $\pp^n$ for some $n\geq 0.$ Then there are locally free sheaves $\mathcal{L}_0,\hdots, \mathcal{L}_n$ on $Y$ such that $X$ is isomorphic over $Y$ to the projection $\pp\left(\bigoplus_{i=0}^n\mathcal{L}_i\right)\rightarrow Y.$
\end{lemma}

\begin{proof}
By Lemma ~\ref{lem:splitting-fans} and ~\cite[Theorem 3.3.19]{CLS11}, the fan $\Sigma_X$ of $X$ can be expressed as a sum $$\Sigma_X=\Sigma_{X,F}+\Sigma_{X,Y}$$
    of subfans $\Sigma_{X,F},\Sigma_{X,Y}\subset \Sigma_X$, where 
    \begin{enumerate}
        \item $\Sigma_{X,F}$ has support equal to ${\rm Ker}(f_*),$
        \item $f_*$ restricts to a bijection $\tau\xrightarrow{\cong}f_*(\tau)$ for each $\tau\in \Sigma_{X,Y},$ 
        \item $f_*(\tau\cap N_X)=f_*(\tau)\cap N_Y$ for each $\tau\in \Sigma_{X,Y}.$
        \item the assignment $\tau\mapsto f_*(\tau)$ determines a bijection $\Sigma_{X,Y}\xrightarrow{\cong}\Sigma_Y.$
    \end{enumerate}
    The homomorphism $f_*\colon N_X\rightarrow N_Y$ induced by $f$ on cocharacter lattices is surjective by (3), so we may choose a section $s\colon N_Y\rightarrow N_X$ of $f_*.$ Since $F\cong \mathbb{P}^n,$ we may choose rays $\sigma_1,\hdots,\sigma_n\in \Sigma_{X,F}^{(1)}$ whose respective primitive generators $v_1,\hdots,v_n$ form a $\zz$-basis for ${\rm Ker}(f_*)\cap N_X.$ It follows that the primitive generator $w_\tau$ of a ray $\tau\in \Sigma_{X,Y}^{(1)}$ can be expressed as $$w_\tau=s\left(f_*(w_\tau)\right)+\sum_{i=1}^na_{\tau,i}v_i$$
    for some integers $a_{\tau,1},\hdots, a_{\tau,n}\in \zz.$ Write $\mathcal{L}_0=\mathcal{O}_Y,$ and write $$\mathcal{L}_i=\mathcal{O}_Y\left(\sum_{\tau\in \Sigma_{X,Y}^{(1)}}-a_{\tau,i}D_{{f_*(\tau)}}\right)$$ for each $1\leq i \leq n,$ where $D_{f_*(\tau)}$ is the torus-invariant divisor on $Y$ corresponding to the ray $f_*(\tau)\in \Sigma_Y^{(1)}.$ It follows from ~\cite[Proposition 7.3.3]{CLS11} that the desired result holds with the locally free sheaves $\mathcal{L}_0,\hdots, \mathcal{L}_n$.
\end{proof}

A given variety may admit multiple different toric log Calabi--Yau pairs. The following lemma explains how these different pairs are related.

\begin{lemma}\label{lem:conjugate-tori}
Let $X$ be a projective variety and let $(X,\Delta_1)$ and $(X,\Delta_2)$ be toric log Calabi--Yau pairs supported on $X.$ Then there is $g\in {\rm Aut}^0(X)$ such that $\Delta_2=g_*\Delta_1.$ If, in addition, $X$ is $\qq$-factorial, then we may write $\Delta_i=\sum_{j=1}^k\Delta_{ij},$ for $1\leq i \leq 2,$ so that $\Delta_{1j}\sim_\qq\Delta_{2j}$ for each $1\leq j \leq k.$
\end{lemma}
\begin{proof}
    $X$ is a projective variety with a rational polyhedral nef cone, so it follows from ~\cite[Corollary 2.12]{brion18} that ${\rm Aut}^0(X)$ is an algebraic group. Writing $n=\dim(X)$, it follows from ~\cite[Lemma 2.2]{Hu18} that subtori of ${\rm Aut}^0(X)$ have dimension at most $n.$ By assumption, the subgroups ${\rm Aut}^0(X,\Delta_1)$ and ${\rm Aut}^0(X,\Delta_2)$ of ${\rm Aut}^0(X)$ are both $n$-dimensional tori. By ~\cite[Theorem 17.10]{Milne17}, there is some $g \in {\rm Aut}^0(X)$ such that ${\rm Aut}^0(X,\Delta_2)=g{\Aut}^0(X,\Delta_1)g^{-1}.$ To obtain the first statement, we recall that, for $1\leq i \leq 2,$ the components of $\Delta_i$ are closures of ${\rm Aut}^0(X,\Delta_i)$-orbits. It follows that, for $1\leq i \leq 2,$ we may write $\Delta_i=\sum_{j=1}^k\Delta_{ij}$ such that $g_*\Delta_{1j}=\Delta_{2j}$ for each $1\leq j \leq k.$ For the second statement, we note that $g_*=(g^{-1})^*$ in this case. Since ${\rm Aut}^0(X)$ is connected, it follows, for all $\qq$-divisors $D,$ that $D$ and $(g^{-1})^*D$ are numerically equivalent. But numerical equivalence and $\qq$-linear equivalence coincide on $X.$ 
\end{proof}


\subsection{Complexity}
In this subsection we recall the notion of complexity and describe its behavior under birational contraction maps.


\begin{definition}\label{defn:complexity}{\em
Let $X$ be a normal projective variety and let $(X,B)$ be a log sub-pair. The \textit{complexity} of $(X,B)$ is $$c(X,B)=\dim X + \rank {\rm WDiv}_{\rm alg}(X)-|B|,$$
where ${\rm WDiv}_{\rm alg}(X)$ is the group of Weil divisors on $X$ modulo algebraic equivalence and $|B|$ is the sum of the coefficients of $B$. 
}
\end{definition}

The following definition describes a variant of the complexity which is natural from the perspective of birational geometry.
\begin{definition}
{\em Let $X$ be a normal projective variety and let $(X,B)$ be a log sub-pair. The \textit{birational complexity} of $(X,B)$ is $$c_{\rm bir}(X,B)=\inf\left\{c(Y,B_Y)\mid (X,B)\simeq _{\rm bir}(Y,B_Y)\right\},$$
where the infimum is taken over all log pairs $(Y,B_Y)$ crepant to $(X,B).$}
\end{definition}



\begin{remark}\label{rmk:compl-zero-implies-logCY}
{\em
Let $X$ be a projective variety and let $(X,B)$ be a log canonical pair with $-(K_X+B)$ nef and $c(X,B)=0.$ Then $(X,B)$ must be a log Calabi--Yau pair. \par
Indeed, it follows from Theorem ~\ref{introthm:BMSZ18} that $X$ is a toric variety, and hence that the nef $\qq$-divisor $-(K_X+B)$ is semi-ample. Thus, there is some $0\leq D\sim_\qq -(K_X+B)$ such that $(X,B+D)$ is log Calabi--Yau. On the one hand, $c(X,B+D)=-|D|$ is nonnegative by Theorem ~\ref{introthm:BMSZ18}. On the other hand, $|D|$ is nonnegative since $D$ is effective. It follows that $|D|=0,$ hence that $D=0,$ and hence that $K_X+B\sim_\qq 0$ as claimed.
}
\end{remark}

\begin{lemma}\label{lem:complexity-and-exceptional-divisors}
Let $f\colon X \dashrightarrow Y$ be a birational contraction map between normal projective varieties. Let $(X,B)$ be a log sub-pair. Denote by $E_1,\hdots, E_r$ the prime $f$-exceptional divisors and write $B_Y=f_*B.$ Then $$c(X,B)=c(Y,B_Y)+\sum_{i=1}^ra_{E_i}(X,B).$$ 
\end{lemma}
\begin{proof}
Choose a resolution 
\[
\xymatrix{ 
 & Z \ar[ld]_-{p} \ar[rd]^-{q} & \\ 
 X \ar@{-->}[rr]^-{f}& & Y
}
\]
of the indeterminacy of $f,$ with $Z$ a smooth projective variety and with $p$ and $q$ birational. Denote by $F_1,\hdots,F_s$ the prime $p$-exceptional divisors and by $\widetilde{E}_1,\hdots,\widetilde{E}_r$ the strict transforms of $E_1,\hdots,E_r,$ respectively. Then $\widetilde{E}_1,\hdots,\widetilde{E}_r,F_1,\hdots,F_s$ are the prime $q$-exceptional divisors. Denote by $(Z,B_Z)$ the log pullback of $(X,B)$ via $p,$ and note that $a_{\widetilde{E}_i}(Z,B_Z)=a_{E_i}(X,B)$ for each $1\leq i \leq r$ (see Remark ~\ref{rmk:crepant-discrepancy}). It follows that the desired equality will hold if we can establish both $$c(Z,B_Z)=c(X,B)+\sum_{i=1}^ra_{\widetilde{E}_i}(Z,B_Z)$$
and $$c(Z,B_Z)=c(Y,B_Y)+\sum_{i=1}^ra_{\widetilde{E}_i}(Z,B_Z)+\sum_{j=1}^sa_{F_j}(Z,B_Z).$$ Thus, we may assume that $f$ is a morphism and that $X$ is smooth.\par
Since $\dim X=\dim Y,$ the desired equality holds if and only if $$\rank {\rm WDiv}_{\rm alg}(X)-|B|=\rank {\rm WDiv}_{\rm alg}(Y)-|B_Y|+\sum_{i=1}^ra_{E_i}(X,B).$$
Since $$|B|=|B_Y|+\sum_{i=1}^r{\rm coeff}_{E_i}(B)$$ and $$\sum_{i=1}^ra_{E_i}(X,B)=r-\sum_{i=1}^r{\rm coeff}_{E_i}(B),$$
this reduces to showing that $$\rank {\rm WDiv}_{\rm alg}(X)=\rank {\rm WDiv}_{\rm alg}(Y)+r.$$ Since $$f_*\colon {\rm WDiv}_{\rm alg}(X)\rightarrow {\rm WDiv}_{\rm alg}(Y)$$ is surjective, it suffices to show that the classes of $E_1,\hdots, E_r$ form a basis for the kernel of $f_*.$\par
All Weil divisors on the smooth variety $X$ are Cartier, and algebraic equivalence as Weil divisors implies numerical equivalence as Cartier divisors in this case. Since the divisors $E_1,\hdots, E_r$ are $f$-exceptional, it then follows from ~\cite[Lemma 3.39]{KM98} that a nonzero divisor of the form $\sum_{i=1}^ra_iE_i$ must be nonzero modulo algebraic equivalence. Thus, the classes of $E_1,\hdots,E_r$ in ${\rm WDiv}_{\rm alg}(X)$ are linearly independent. These classes are certainly contained in the kernel of $f_*,$ and the fact that they generate the kernel of $f_*$ follows from the fact that we may identify this kernel with the kernel of the restriction homomorphism ${\rm WDiv}_{\rm alg}(X)\rightarrow {\rm WDiv}_{\rm alg}\left(X\setminus {\rm Ex}(f)\right).$
\end{proof} 

As a corollary, we see that the complexity of a log canonical pair is unaffected by extracting log canonical places.

\begin{corollary}\label{cor:compl-dlt-mod}
    Let $(X,B)$ be a log canonical pair and let $f\colon (Y,B_Y)\rightarrow (X,B)$ be a birational morphism extracting only log canonical places of $(X,B)$. Then $$c(X,B)=c(Y,B_Y).$$
\end{corollary}
\begin{proof}
    Denote by $E_1,\hdots, E_r$ the prime $f$-exceptional divisors. By assumption, $a_{E_i}(Y,B_Y)=0$ for all $1\leq i \leq r.$ The desired result follows from Lemma ~\ref{lem:complexity-and-exceptional-divisors}.
\end{proof}

As another corollary, we obtain the following special case of Theorem ~\ref{introthm:exceptional-loci}.

\begin{corollary}\label{cor:complexity-zero-and-exceptional-divisors}
Let $f\colon X \dashrightarrow Y$ be a birational contraction map between normal projective varieties and let $(X,B)$ be a log Calabi--Yau pair of complexity zero. Then $(Y,f_*B)$ is a log Calabi--Yau pair of complexity zero, and every $f$-exceptional divisor is a component of $\lfloor B \rfloor.$
\end{corollary}
\begin{proof}
Denote by $E_1,\hdots, E_r$ the prime $f$-exceptional divisors, and write $B_Y=f_*B.$ Since $f_*$ is a birational contraction and $(X,B)$ is a log Calabi--Yau pair, it follows that $(Y,B_Y)$ is a log Calabi--Yau pair. Thus, on the one hand, Theorem ~\ref{introthm:BMSZ18} implies that $c(Y,B_Y)\geq 0.$ On the other hand, it follows from Lemma ~\ref{lem:complexity-and-exceptional-divisors} and the assumption $c(X,B)=0$ that $$c(Y,B_Y)=-\sum_{i=1}^ra_{E_i}(X,B).$$ 
The log discrepancies $a_{E_i}(X,B)$ are nonnegative since $(X,B)$ is log canonical, so this is possible only if $a_{E_i}(X,B)=0$ for each $1\leq i \leq r.$ It follows that ${\rm coeff}_{E_i}(B)=1$ for each $1\leq i \leq r$ and that $c(Y,B_Y)=0.$
\end{proof}

\subsection{Degenerate divisors}
We recall the following definitions from ~\cite{Lai11}:
\begin{definition}\label{defn:degn-divs} {\em
Let $f\colon X\rightarrow Y$ be a proper surjective morphism of normal varieties and let $D\in {\rm WDiv}_\qq(X)$ be effective. We say that $D$ is:
\begin{itemize}
    \item \textit{$f$-exceptional} if ${\rm codim}({\rm Supp}(f(D)))\geq 2,$
    \item \textit{of insufficient fiber type} if ${\rm codim}({\rm Supp}(f(D))=1$ and there exists a prime divisor $\Gamma \nsubset {\rm Supp}(D)$ such that $f(\Gamma)\subset {\rm Supp}(f(D))$ has codimension one in $Y.$
\end{itemize}
In either of the above cases, we say that $D$ is \textit{degenerate}. In particular, degenerate divisors are always assumed to be effective.}
\end{definition}
The following appears as ~\cite[Lemma 2.10]{Lai11}.
\begin{lemma}\label{lem:degn-divs-neg}
Let $f\colon X \rightarrow Y$ be a fibration between normal projective varieties with $X$ $\qq$-factorial. Let $D$ be a degenerate divisor on $X.$ Then there is a component $\widetilde{D}\subset {\rm Supp}(D)$ which is covered by curves contracted by $f$ and intersecting $\widetilde{D}$ negatively.
\end{lemma}

Applying this to the case of degenerate prime divisors, we obtain the following.

\begin{corollary}\label{cor:contract-degn-divs}
Let $f\colon X \rightarrow Y$ be a fibration between normal projective varieties with $X$ $\qq$-factorial. Let $D\subset X$ be a degenerate prime divisor, and assume that there is a $D$-MMP over $Y$ that terminates. Then there is a birational contraction map $X\dashrightarrow X'$ over $Y$ whose only exceptional divisor is $D.$
\end{corollary}

\begin{lemma}\label{lem:compl-zero-degn-divs}
Let $f\colon X \rightarrow Y$ be a fibration between $\qq$-factorial projective varieties, and let $(X,B)$ be a log Calabi--Yau pair of complexity zero. Then every degenerate divisor on $X$ is contained in $\lfloor B \rfloor.$
\end{lemma}
\begin{proof}
    Suppose $D\subset X$ is a degenerate prime divisor. It follows from Theorem ~\ref{introthm:BMSZ18} that $X$ is a Mori dream space, and so it follows from Corollary ~\ref{cor:contract-degn-divs} that there exists a birational contraction map $X\dashrightarrow X'$ over $Y$ whose only exceptional divisor is $D.$ Finally, it follows from Corollary ~\ref{cor:complexity-zero-and-exceptional-divisors} that $D$ is a component of $\lfloor B \rfloor.$
\end{proof}

\subsection{Canonical bundle formula}
In this subsection we recall the canonical bundle formula, emphasizing the simple form it takes in the special case of toric log Calabi--Yau pairs.

\begin{definition}\label{def:can-bun-formula}{\em
Let $f\colon X \rightarrow Y$ be a contraction with $\dim Y>0$ and let $(X,B)$ be a log canonical pair with $K_X+B\sim_{\qq,f}0.$ This data determines a \textit{discriminant b-divisor} $\mathbf{B}$ and a \textit{moduli b-divisor} $\mathbf{M}$ on $Y$ (see ~\cite[Section 3.4]{FG14b}). We will refer to $(Y,B_Y,\mathbf{M})$ as the \textit{generalized pair determined by the canonical bundle formula.}
}
\end{definition}
We refer the reader to ~\cite{MM24} for details about generalized pairs and their singularities.
\begin{remark}{\em
Notation as in Definition ~\ref{def:can-bun-formula}. The trace $B_Y$ of $\mathbf{B}$ on $Y$ can be described as follows. For each prime divisor $D\subset Y,$ write $${\rm lct}_{D}(X,B; f^*D)=\max\{t\in \qq\vert (X,B+tf^*D)\text{ is lc over the generic point of }D \}.$$Then $B_Y$ satisfies $${\rm coeff}_D(B_Y)=1-{\rm lct}_{D}(X,B; f^*D)$$ for each prime divisor $D\subset Y.$ The trace $M_Y$ of $\mathbf{M}$ on $Y$ is characterized up to $\qq$-linear equivalence by the property $$K_X+B\sim_\qq f^*(K_Y+B_Y+M_Y),$$
known as the \textit{canonical bundle formula}.
}
\end{remark}
\begin{lemma}
\label{lem:follow-fibr}
Let $f\colon X \rightarrow Y$ be a fibration between $\qq$-factorial varieties of Fano type. Let $(X,B)$ be a log Calabi--Yau pair. Let $(Y,B_Y,\mathbf{M})$ be the generalized pair determined by the canonical bundle formula. Let $\phi_Y\colon Y'\dashrightarrow Y$ be a birational map
between $\qq$-factorial varieties extracting only glc places of $(Y,B_Y,\mathbf{M})$.
Then, there exists a commutative diagram
\[
\xymatrix{ 
(X,B)\ar[d]_-{f} & (X',B')\ar@{-->}[l]_-{\phi}\ar[d]^-{f'} \\ 
(Y,B_Y,\mathbf{M}) &
(Y',B_{Y'},\mathbf{M})\ar@{-->}[l]^-{\phi_Y}
}
\]
satisfying the following conditions:
\begin{enumerate}
\item $X'$ is $\qq$-factorial,
\item $\phi$ is a crepant birational map extracting only log canonical places of $(X,B)$, 
\item $f'$ is a fibration, and
\item $f'$ is extremal if $f$ is.
\end{enumerate}
\end{lemma}
\begin{proof}
All but (3) follow from the statement of ~\cite[Lemma 2.12]{MM24}, and (3) follows from the proofs provided for ~\cite[Lemmas 2.10-2.12]{MM24}. 
\end{proof}

We will use Lemma ~\ref{lem:follow-fibr} together with the following result.

\begin{lemma}\label{lem:toric-canonical-bundle-formula}
    Let $X$ be a toric variety and let $f\colon X \rightarrow Y$ be a contraction. Write $\Delta$ and $\Delta_Y$ for the toric boundaries of $X$ and $Y,$ respectively. Let $(Y,B_Y,\mathbf{M})$ be the generalized pair determined by $(X,\Delta)$ and $f$ via the canonical bundle formula. Then $B_Y=\Delta_Y$ and $\mathbf{M}\sim 0$ where it descends.
\end{lemma}
\begin{proof}
This follows from ~\cite[Lemma 2.4]{BirChen21}.
\end{proof}

\section{Toric boundary arrangements}
In this section we prove Theorem~\ref{introthm:convex-toric-div}. To do this, we study toric log Calabi--Yau pairs associated to log Calabi--Yau pairs of complexity zero.

\subsection{Associated toric divisors}
In this subsection we define two invariants and prove some lemmas regarding these invariants and the set of associated toric divisors to a log Calabi--Yau pair of complexity zero.

We begin this subsection by restating the following definition.

\begin{definition}\label{def:associated-divisors}{\em
Let $(X,B)$ be a log pair. We say that a Weil divisor $\Delta$ on $X$ is \textit{associated} to $(X,B)$ if the following conditions hold:
\begin{enumerate}
\item $(X,\Delta)$ is a toric log Calabi--Yau pair,
\item $\lfloor B \rfloor \leq \Delta \leq \lceil B \rceil.$
\end{enumerate}
We write 
\[
\mathcal{A}(X,B)=\left\{\Delta\in \text{WDiv}(X)\mid \Delta\text{ is associated to }(X,B)\right\}.
\]
Note that this is a finite set. Given a set $\mathcal{V}$ of log canonical places of $(X,B),$ we write 
\[
\mathcal{A}_\mathcal{V}(X,B)=\left\{\Delta\in \mathcal{A}(X,B)\mid E\text{ is a log canonical place of }(X,\Delta)\text{ for all }E\in \mathcal{V}\right\}.
\]
}
\end{definition}


\begin{definition}{\em 
Let $(X,B)$ be a log Calabi--Yau pair of complexity zero, and let $\Delta \in \mathcal{A}(X,B)$. We define the following invariant:
\[
\lambda_1(X,B;\Delta):=\max\left\{\lambda \in [0,1]\mid \lambda\Delta \leq B\right\}.
\]
Whenever $\lambda_1(X,B;\Delta)<1,$ we will also define:
\[
\lambda_2(X,B;\Delta):=\max\left\{\lambda \in \left[0,\lambda_1(X,B;\Delta)\right] \,\middle|\, \left(X,\frac{1}{1-\lambda}(B-\lambda \Delta)\right) \text{ is log canonical}\right\}.
\]
When $\lambda_1(X,B;\Delta)=1,$ we will set $\lambda_2(X,B;\Delta)=1.$ When $(X,B)$ and $\Delta$ are clear from context, we will simply write $\lambda_1$ and $\lambda_2.$
}
\end{definition}

\begin{remark}\label{rem:rat-lambda}
{\em The invariants defined above are always rational numbers. Indeed, $\lambda_1(X,B;\Delta)$ is the smallest coefficient in $B$ of a component of the support of $\Delta$, and the rationality of $\lambda_2(X,B;\Delta)$ can be seen by computing it on a log resolution of $(X,\lceil B \rceil).$}
\end{remark}

The following lemma indicates the significance of these invariants.

\begin{lemma}\label{lem:lambda-lcps}
Let $(X,B)$ be a log Calabi--Yau pair of complexity zero, and let $\Delta\in \mathcal{A}(X,B)$ be an associated divisor. Then the following hold:
\begin{enumerate}
    \item If $\lambda_2=0,$ then there is a log canonical place of $(X,B)$ which is not a log canonical place of $(X,\Delta).$
    \item If $\lambda_2<\lambda_1<1,$ then there is a log canonical place of $\left(X,\frac{1}{1-\lambda_2}(B-\lambda_2 \Delta)\right)$ which is not a log canonical place of $(X,\Delta).$
\end{enumerate}
\end{lemma}

\begin{proof}
To show (1), we suppose on the contrary that every log canonical place of $(X,B)$ is a log canonical place of $(X,\Delta)$. We will show that the pair $(X,\frac{1}{1-\lambda}(B-\lambda\Delta))$ is log canonical for all sufficiently small $\lambda>0$, hence that $\lambda_2>0.$ Fix a log resolution $Y \rightarrow X$ of $(X,B),$ hence also of $(X,\Delta),$ and denote by $(Y,B_Y)$ and $(Y,\Delta_Y)$ the log pullbacks of $(X,B)$ and $(X,\Delta),$ respectively. Note that ${\rm coeff}_E(B_Y)\leq 1$ for all prime divisors $E\subset Y$ since $(X,B)$ is log canonical.\par 
Given a divisor $E$ on $Y,$ we have
$${\rm coeff}_E\left(\frac{1}{1-\lambda}(B_Y-\lambda\Delta_Y)\right)=\frac{1}{1-\lambda}\left({\rm coeff}_E(B_Y)-\lambda{\rm coeff}_E(\Delta_Y)\right)$$
for all $\lambda<1.$ These coefficients are continuous functions in $\lambda,$ and we recover the coefficients of $B_Y$ when $\lambda=0.$ When ${\rm coeff}_E(B_Y)=1$, we have ${\rm coeff}_E(\Delta_Y)=1$ since every log canonical place of $(X,B)$ is a log canonical place of $(X,\Delta).$ In this case, it follows that ${\rm coeff}_E(\frac{1}{1-\lambda}(B_Y-\lambda\Delta_Y))=1$ for all $\lambda<1.$ When ${\rm coeff}_E(B_Y)<1$, we have ${\rm coeff}_E\left(\frac{1}{1-\lambda}(B_Y-\lambda\Delta_Y)\right)<1$ for all sufficiently small $\lambda>0$ by continuity. Since there are only finitely many $E$ that have nonzero coefficient in at least one of $B_Y$ or $\Delta_Y$, it follows that $(X,\frac{1}{1-\lambda}(B-\lambda\Delta))$ is log canonical for all sufficiently small $\lambda>0.$\par
We now turn to show (2). Write $B'=\frac{1}{1-\lambda_2}(B-\lambda_2 \Delta),$ and note that the condition $\lambda_2<\lambda_1<1$ implies that $\Delta\in \mathcal{A}(X,B').$ Thus, it suffices by (1) to show that $\lambda_2\left(X,B';\Delta\right)=0.$ For all $\lambda\in [0,1),$ we have $$\frac{B'-\lambda\Delta}{1-\lambda}=\frac{B-(\lambda_2+\lambda(1-\lambda_2))\Delta}{1-(\lambda_2+\lambda(1-\lambda_2))}.$$
But $\lambda_2+\lambda(1-\lambda_2)>\lambda_2$ whenever $\lambda>0,$ so it follows from the definition of $\lambda_2$ that $\left(X,\frac{1}{1-\lambda}(B'-\lambda \Delta)\right)$ is not log canonical for any $\lambda\in (0,1).$
\end{proof}

These notions provide several ways to characterize toric log Calabi--Yau pairs:

\begin{lemma}\label{lem:compl-zero-ind-one}
Let $(X,B)$ be a log Calabi--Yau pair of complexity zero. The following are equivalent:
\begin{enumerate}
    \item $(X,B)$ is a toric log Calabi--Yau pair,
    \item $(X,B)$ has index one,
    \item there exists $\Delta\in \mathcal{A}(X,B)$ with $\lambda_1(X,B;\Delta)=1.$
\end{enumerate}
\end{lemma}
\begin{proof}

All toric log Calabi--Yau pairs have index one, as shown in ~\cite[Section 4.1]{Rei83}. Now assume that $(X,B)$ has index one. It follows that $B$ must have integer coefficients. By Corollary ~\ref{introcor:BMSZ18}, there exists $\Delta \in \mathcal{A}(X,B)$. We have $\Delta\leq\lceil B\rceil=B,$ from which it follows that $\lambda_1(X,B;\Delta)=1.$ Finally, assume that there exists some $\Delta\in \mathcal{A}(X,B)$, with  $\lambda_1(X,B;\Delta)=1$. By definition of $\lambda_1,$ we have $\Delta\leq B.$ Thus, $B-\Delta$ is an effective divisor. Since $K_X+B \sim_{\mathbb{Q}}K_X+\Delta$, it follows that $B-\Delta \sim_{\mathbb{Q}} 0$ and hence that $B=\Delta$. Thus, $(X,B)$ is a toric log Calabi--Yau pair.

\end{proof}

Outside of this case, we have the following:

\begin{lemma}\label{lem:peeling-inclusion}
Let $(X,B)$ be a log Calabi--Yau pair of complexity zero, and let $\Delta \in \mathcal{A}(X,B)$ have $\lambda_1<1.$ Then for all $\lambda \in [0,\lambda_2],$ $$\mathcal{A}\left(X,\frac{1}{1-\lambda}(B-\lambda \Delta)\right)\subseteq \mathcal{A}(X,B).$$

Furthermore, if $\lambda_2=\lambda_1<1,$ then this containment is strict.
\end{lemma}

\begin{proof}
Given $\lambda \in [0,\lambda_2],$ write $B_\lambda=\frac{1}{1-\lambda}(B-\lambda \Delta).$ 
The condition $\lfloor B \rfloor\leq \Delta$ implies that any divisor appearing in $B$ with coefficient $1$ must also appear in $B_\lambda$ with coefficient $1$. In other words, we must have $\lfloor B \rfloor \leq \lfloor B_\lambda \rfloor.$ The condition $\Delta\leq \lceil B \rceil$ implies that any divisor appearing in $B$ with coefficient $0$ must also appear in $B_\lambda$ with coefficient $0.$ In other words, we must have $\lceil B_\lambda \rceil\leq \lceil B \rceil.$ So given any $\Gamma \in \mathcal{A}(X,B_\lambda),$ it follows from $$\lfloor B \rfloor \leq \lfloor B_\lambda \rfloor\leq \Gamma \leq \lceil B_\lambda \rceil\leq \lceil B \rceil$$ that $\Gamma \in \mathcal{A}(X,B).$

From now on, assume that $\lambda_2=\lambda_1<1$. By the definition of $\lambda_1$, we have that there exists a prime divisor $E$ in the support of $B$ and $\Delta$, such that, ${\rm coeff}_E(B)=\lambda_1=\lambda_2$. Thus, we have
$${\rm coeff}_E\left(\frac{1}{1-\lambda_2}(B-\lambda_2 \Delta)\right)=0.$$
Therefore $\Delta \notin \mathcal{A}\left(X,\frac{1}{1-\lambda_2}(B-\lambda_2 \Delta)\right)$, showing that the containment is strict.
\end{proof}

\begin{lemma}\label{lem:extraction-inclusion}
Let $(X,B)$ be a log Calabi--Yau pair of complexity zero, and let $f\colon Y \rightarrow X$ be a projective birational morphism extracting only log canonical places of $(X,B).$ Write $(Y,B_Y)$ for the log pullback of $(X,B)$ via $f,$ and write $\mathcal{V}$ for the set of $f$-exceptional divisors. Then pushforward along $f$ induces a bijection $$\mathcal{A}(Y,B_Y)\xrightarrow{\cong} \mathcal{A}_\mathcal{V}(X,B).$$
\end{lemma}

\begin{proof}
Given any toric log Calabi--Yau pair $(Y, \Gamma),$ it follows from Lemma ~\ref{lem:toric-bir-contr} that the pair $(X,f_*\Gamma)$ is also toric log Calabi--Yau and that $f\colon (Y, \Gamma)\rightarrow (X,f_*\Gamma)$ is a crepant birational morphism between these pairs. It is clear, therefore, that $f_*\mathcal{A}(Y,B_Y)\subset \mathcal{A}_\mathcal{V}(X,B).$ For surjectivity, consider some $\Delta \in \mathcal{A}_\mathcal{V}(X,B).$ Denote by $(Y,\Delta_Y)$ the log pullback of $(X,\Delta)$ to $Y$. Then $(Y,\Delta_Y)$ is a log Calabi--Yau pair of index one since $(X,B)$ is, and 
    $$c(Y,\Delta_Y)=c(X,\Delta)=0$$
since every $f$-exceptional divisor is a log canonical place for $(X,\Delta).$ It follows from Lemma ~\ref{lem:compl-zero-ind-one} that $(Y,\Delta_Y)$ is a toric log Calabi--Yau pair, and it follows from the fact that every $f$-exceptional divisor appears in $B_Y$ with coefficient $1$ that
$$\lfloor B_Y\rfloor\leq \Delta_Y\leq \lceil B_Y \rceil.$$ We see that $\Delta_Y$ is an element of $\mathcal{A}(Y,B_Y)$ satisfying $f_*\Delta_Y=\Delta.$ For injectivity, we note that two divisors $\Gamma_1,\Gamma_2 \in {\rm WDiv}(Y)$ satisfying $f_*\Gamma_1=f_*\Gamma_2$ can differ only at $f$-exceptional divisors. But since every $f$-exceptional divisor appears in $B_Y$ with coefficient $1,$ they must all appear in every element of $\mathcal{A}(Y,B_Y)$ with coefficient $1$ as well.

\end{proof}

\begin{proposition}\label{prop:one-ass-pair}
Let $(X,B)$ be a log Calabi--Yau pair of complexity zero. Then $(X,B)$ is a toric log Calabi--Yau pair if and only if $|\mathcal{A}(X,B)|=1$.

\end{proposition}

\begin{proof}
First, suppose that $(X,B)$ is a toric log Calabi--Yau pair. In particular, $B$ has integer coefficients, hence $\lfloor B \rfloor =B = \lceil B \rceil$. It follows that any divisor $\Delta \in \mathcal{A}(X,B)$ must satisfy $B=\lfloor B \rfloor \leq \Delta \leq \lceil B \rceil=B$, hence $\mathcal{A}(X,B)=\{B\}$.\par

Conversely, suppose that $|\mathcal{A}(X,B)|=1.$ Denote by $\Delta$ be the unique divisor associated to $(X,B)$. If $\lambda_1=1$, then we are done by Lemma~\ref{lem:compl-zero-ind-one}. So assume, for a contradiction, that $\lambda_1<1.$
Thus, $\frac{1}{1-\lambda}(B-\lambda \Delta)$ is a nonzero effective divisor for all $\lambda\in [0,\lambda_1]$. If $\lambda_2=\lambda_1<1$, then it would follow from Lemma~\ref{lem:peeling-inclusion} that   $\mathcal{A}\left(X,\frac{1}{1-\lambda_2}(B-\lambda_2 \Delta)\right)$ is empty, contradicting Corollary ~\ref{introcor:BMSZ18}. \par 
From now on we assume that $\lambda_2<\lambda_1<1$. Denote by $B':=\frac{1}{1-\lambda_2}(B-\lambda_2 \Delta)$, and note that we must have $\mathcal{A}(X,B')=\{\Delta\}$ by Lemma~\ref{lem:peeling-inclusion}. It follows from part (2) of Lemma ~\ref{lem:lambda-lcps} that there is a log canonical place $E$ of $(X,B')$ that is not a log canonical place of $(X,\Delta).$ By \cite[Theorem 1]{Mor20}, there is a projective birational morphism $f\colon Y\rightarrow X$ with divisorial exceptional locus which extracts only the divisor $E.$ Denote by $(Y,B'_Y)$ the log pullback  of $(X,B')$ via $f.$ It follows from Lemma ~\ref{lem:extraction-inclusion} that $f_*\mathcal{A}(Y,B_Y')=\mathcal{A}_{\{E\}}(X,B'),$ but this set is empty since $E$ is not a log canonical place of $(X,\Delta).$ It would then have to follow that $\mathcal{A}(Y,B_Y')=\emptyset,$ contradicting Corollary ~\ref{introcor:BMSZ18}.

\end{proof}

\subsection{Toric boundary arrangements}

We start this section with the following definition.

\begin{definition}\label{def:toric-bound-arrang}{\em 
We say that a log pair $(X,B)$ is a \textit{toric boundary arrangement} if we can write $B=\sum_{i=1}^rb_i\Delta_i,$ where:
\begin{enumerate}
\item $\Delta_1,\hdots, \Delta_r\in \mathcal{A}(X,B),$ 
\item $b_1,\hdots, b_r$ are nonnegative and satisfy $\sum_{i=1}^rb_i=1.$
\end{enumerate}
It follows from the definition that a toric boundary arrangement is, in particular, a log Calabi--Yau pair of complexity zero.
}
\end{definition}

\begin{lemma}\label{lem:peeling-blend}
Let $(X,B)$ be a log Calabi--Yau pair of complexity zero, and let $\Delta\in \mathcal{A}(X,B)$ be such that $\lambda_1<1.$ If $\left(X,\frac{1}{1-\lambda_2}(B-\lambda_2 \Delta)\right)$ is a toric boundary arrangement, then so is $(X,B).$
\end{lemma}

\begin{proof}
Write $B_{\lambda_2}=\frac{1}{1-\lambda_2}(B-\lambda_2 \Delta).$ If $(X,B_{\lambda_2})$ is a toric boundary arrangement, then there are $\Gamma_1,\hdots, \Gamma_r\in \mathcal{A}(X,B_{\lambda_2})$ and nonnegative $b_1,\hdots, b_r$ satisfying $\sum_{i=1}^rb_i=1$ such that $\sum_{i=1}^rb_i\Gamma_i=B_{\lambda_2}.$ But then we have $$B=\lambda_2\Delta+\sum_{i=1}^rb_i(1-\lambda_2)\Gamma_i.$$
Since $\mathcal{A}(X,B_{\lambda_2})\subseteq \mathcal{A}(X,B)$ by Lemma ~\ref{lem:peeling-inclusion}  and since $\lambda_2,b_1(1-\lambda_2),\hdots, b_r(1-\lambda_2)$ are nonnegative and satisfy $\lambda_2+\sum_{i=1}^rb_i(1-\lambda_2)=1$, it follows that $(X,B)$ is a toric boundary arrangement.
\end{proof}

\begin{lemma}\label{lem:extraction-blend}
Let $(X,B)$ be a log Calabi--Yau pair of complexity zero, and let $f\colon Y \rightarrow X$ be a projective birational morphism extracting only log canonical places of $(X,B).$ Write $(Y,B_Y)$ for the log pullback of $(X,B)$ via $f.$ If $(Y,B_Y)$ is a toric boundary arrangement, then so is $(X,B).$
\end{lemma}

\begin{proof}
If $(Y,B_Y)$ is a toric boundary arrangement, then there are $\Gamma_1,\hdots, \Gamma_r\in \mathcal{A}(Y,B_Y)$ and nonnegative $b_1,\hdots, b_r$ satisfying $\sum_{i=1}^rb_i=1$ such that $\sum_{i=1}^rb_i\Gamma_i=B_Y.$ It follows from Lemma ~\ref{lem:extraction-inclusion} that $f_*\Gamma_1,\hdots, f_*\Gamma_r\in \mathcal{A}(X,B).$ Since $B=f_*B_Y,$ we have $$B=\sum_{i=1}^rb_if_*\Gamma_i,$$ hence that $(X,B)$ is a toric boundary arrangement.
\end{proof}

\begin{proof}[Proof of Theorem~\ref{introthm:convex-toric-div}]

We induct on the cardinality of the set $\mathcal{A}(X,B).$ As previously noted, it follows from Corollary ~\ref{introcor:BMSZ18} that this set is nonempty and finite.\par 
By Proposition~\ref{prop:one-ass-pair}, if $|\mathcal{A}(X,B)|=1$, then $(X,B)$ is a toric log Calabi--Yau pair and the desired result holds trivially. From now on we assume that $|\mathcal{A}(X,B)|>1.$ \par
Choose any $\Delta\in \mathcal{A}(X,B)$. It follows from Lemma ~\ref{lem:compl-zero-ind-one} and Proposition ~\ref{prop:one-ass-pair} that $\lambda_1<1$ and hence that $\lambda_2<1$. Set $B':=\frac{1}{1-\lambda_2}(B-\lambda_2 \Delta)$. By definition of $\lambda_2,$ we have that $(X,B')$ is a log Calabi--Yau pair of complexity zero. If $\lambda_2=\lambda_1$, then it follows from Lemma~\ref{lem:peeling-inclusion} that $|\mathcal{A}(X,B')|<|\mathcal{A}(X,B)|.$ Thus, $(X,B')$ is a toric boundary arrangement by the inductive hypothesis, and it then follows from  Lemma ~\ref{lem:peeling-blend} that $(X,B)$ is a toric boundary arrangement.\par
From now on, we assume that $\lambda_2<\lambda_1<1$. It follows from part (2) of Lemma ~\ref{lem:lambda-lcps} that there is a log canonical place $E$ of $(X,B')$ that is not a log canonical place of $(X,\Delta).$ By \cite[Theorem 1]{Mor20}, there is a normal, $\qq$-factorial projective variety $Y$ and a projective birational morphism $f\colon Y\rightarrow X$ with divisorial exceptional locus which extracts only the divisor $E.$ Denote by $(Y,B'_Y)$ the log pullback of $(X,B')$ via $f.$ It follows from Lemma ~\ref{lem:extraction-inclusion} and the fact that $E$ is not a log canonical place for $(X,\Delta)$ that $|\mathcal{A}(Y,B_Y')|<|\mathcal{A}(X,B')|.$ But $|\mathcal{A}(X,B')|\leq|\mathcal{A}(X,B)|$ by Lemma ~\ref{lem:peeling-inclusion}, and so it follows by the inductive hypothesis that $(Y,B_Y')$ is a toric boundary arrangement. Lemma ~\ref{lem:extraction-blend} implies that $(X,B')$ is a toric boundary arrangement, and Lemma ~\ref{lem:peeling-blend} then implies that $(X,B)$ is a toric boundary arrangement. 
\end{proof}

\section{Geometry of log canonical centers}

\begin{proof}[Proof of Theorem~\ref{introthm:lccs}]
    We begin by using Theorem ~\ref{introthm:convex-toric-div} to express $(X,B)$ as a toric boundary arrangement with $B=\sum_{i=1}^rb_i\Delta_i.$\par
    In the case of (1), it follows that $E$ is a log canonical place for each of the pairs $(X,\Delta_i).$ By the linearity of discrepancy with respect to the boundary, it follows that $E$ is a log canonical place for $(X,B).$ In the case of (2), use \cite[Theorem 1]{Mor20} to obtain a projective birational morphism $f\colon Y\rightarrow X$ which extracts only the divisor $E.$ Write $(Y,B_Y)$ for the log pullback of $(X,B)$ via $f.$ By Lemma ~\ref{lem:extraction-inclusion}, the set $\mathcal{A}_{\{E\}}(X,B)$ is in bijection with $\mathcal{A}(Y,B_Y)$. This latter set is nonempty by Theorem ~\ref{introthm:BMSZ18}, since $(Y,B_Y)$ is a log Calabi--Yau pair of complexity zero. \par
    In the case of (3), let $f\colon Y\rightarrow X$ be the normalization of the blow up of $X$ along $Z.$ For each $1\leq i \leq r,$ choose a maximal torus $\mathbb{T}_i\leq {\rm Aut(X,\Delta_i)}$. For each $1\leq i \leq r,$ $Z$ is $\mathbb{T}_i$-invariant and the action of $\mathbb{T}_i$ lifts to an action on $Y$ such that $f$ is $\mathbb{T}_i$-equivariant. Moreover, $f$ is an isomorphism over the big torus $\mathbb{T}_i=X\setminus\Delta_i$ for each $1\leq i \leq r.$ It follows that each of the actions of $\mathbb{T}_1,\hdots,\mathbb{T}_r$ on $Y$ endows $Y$ with the structure of a toric variety; denote by $\Gamma_i$ the toric boundary corresponding to the action of $\mathbb{T}_i.$ Then $f\colon (Y,\Gamma_i)\rightarrow (X,\Delta_i)$ is a crepant projective birational morphism for each $1\leq i \leq r.$ Note that, for each $1\leq i \leq r,$ each $f$-exceptional divisor must be a component of $\Gamma_i=Y\setminus\mathbb{T}_i,$ hence a log canonical place of $(X,\Delta_i).$ It follows that each $f$-exceptional divisor is a log-canonical place of $(X,B),$ hence that $Z$ is a log-canonical center of $(X,B).$\par
    In the case of (4), choose some log-canonical place $E$ of $(X,B)$ whose center on $X$ is $Z.$ By Part (2) of this theorem, there is some $\Delta\in \mathcal{A}(X,B)$ with respect to which $E$ is toric. It follows that the center $Z$ of $E$ on $X$ is a stratum of this $\Delta.$
\end{proof}

\begin{corollary}\label{cor:lcps-from-sing}
    Let $(X,B)$ be a log Calabi--Yau pair of complexity zero and let $Z\subset {\rm Sing}(X)$ be an irreducible component of the singular locus of $X.$ Let $f\colon Y\rightarrow X$ be the normalized blow-up of $X$ along $Z$ and let $(Y,B_Y)$ be the log pullback of $(X,B)$ via $f$. Then every $f$-exceptional divisor is a log canonical place of $(X,B).$
\end{corollary}
\begin{proof}
    Let $\Delta \in \mathcal{A}(X,B)$ and write $\mathbb{T}={\rm Aut}^0(X,B)$. As an irreducible component of ${\rm Sing}(X),$ $Z$ must be a toric stratum of $(X,\Delta).$ Since $f\colon Y \rightarrow X$ is the normalization of the blowing up of a $\mathbb{T}$-invariant subvariety, it follows that the $\mathbb{T}$-action on $X$ lifts to a $\mathbb{T}$-action on $Y.$ Since $Z\subset \Delta,$ it follows that $f$ is an isomorphism over the open orbit $X\setminus \Delta$ and hence that the $\mathbb{T}$-action on $Y$ has an open orbit with trivial isotropy group. Writing $\Gamma$ for the reduced sum of the divisorial components of the complement of this orbit, we obtain a toric log Calabi--Yau pair $(Y,\Gamma)$ with the property that $f_*\Gamma=\Delta.$ In particular, $(Y,\Gamma)$ is the log pullback of $(X,\Delta)$ via $f.$ It follows from Corollary ~\ref{cor:complexity-zero-and-exceptional-divisors} that each $f$-exceptional divisor is a component of $\Gamma.$ Thus, each $f$-exceptional divisor is toric with respect to $(X,\Delta).$ The desired result now follows from Part (1) of Theorem ~\ref{introthm:lccs}.
\end{proof}

\begin{proof}[Proof of Theorem~\ref{introthm:dlt-log-smth}]
    Since $(Y,B_Y)$ is a log Calabi--Yau pair of complexity zero and $\lfloor B_Y \rfloor\leq \Gamma$ for all $\Gamma\in \mathcal{A}(Y,B_Y),$ it suffices to show that $Y$ is smooth. For this, we note that it follows from Corollary ~\ref{cor:lcps-from-sing} that any irreducible component of ${\rm Sing}(Y)$ would be a log canonical center of $(Y,B_Y).$ But each log canonical center of the dlt pair $(Y,B_Y)$ must intersect the smooth locus of $Y,$ so we must have ${\rm Sing}(Y)=\emptyset.$
\end{proof}

\section{Behavior with respect to contractions}

\begin{proof}[Proof of Theorem~\ref{introthm:exceptional-loci}]
It follows from Corollary ~\ref{cor:complexity-zero-and-exceptional-divisors} that $(Y,B_Y=f_*B)$ is a log Calabi--Yau pair of complexity zero. To show that the exceptional locus of $f$ is a union of log canonical centers of $(X,B),$ it suffices by part (3) of Theorem ~\ref{introthm:lccs} to show that each irreducible component of the exceptional locus is a toric stratum of every $\Delta\in \mathcal{A}(X,B).$ But this follows from Lemma ~\ref{lem:toric-bir-contr}.
\end{proof}

\begin{proof}[Proof of Theorem~\ref{introthm:canonical-bundle-formula}]
Let $D\subset Y$ be a prime divisor. First, we consider the case that $D$ is not a component of $\Gamma_i$ for each $1\leq i \leq r.$ Using Lemma ~\ref{lem:toric-canonical-bundle-formula} to identify $\Gamma_i$ with the discriminant divisor induced by $f$ and $\Delta_i,$ it follows that ${\rm lct}_{D}(X,\Delta_i;f^*D)=1$ for each $1\leq i \leq r.$ Since $$B+tf^*D=\sum_{i=1}^rb_i(\Delta_i+tf^*D)$$
for all $t\in\qq,$ it follows that ${\rm lct}_D(X,B;f^*D)=1$
and hence that ${\rm coeff}_D(B_Y)=0.$\par
Next, we consider the case that $D$ is a component of $\Gamma_i$ for at least one $1\leq i \leq r$ and that $f^{-1}(D)$ contains more than one prime divisor. In this case, each prime divisor contained in $f^{-1}(D)$ is degenerate over $Y.$ By Lemma ~\ref{lem:compl-zero-degn-divs}, each of these divisors is a component of $\lfloor B \rfloor.$ On the one hand, this implies that ${\rm lct}_{D}(X,B;f^*D)=0,$ hence that ${\rm coeff}_D(B_Y)=1.$ On the other hand, this implies that each prime divisor contained in $f^{-1}(D)$ is a component of $\Delta_i$ for each $1\leq i \leq r,$ hence that $D$ is a component of $\Gamma_i$ for each $1\leq i \leq r.$ Thus, we have ${\rm coeff}_D(B_Y)={\rm coeff}_D(\sum_{i=1}^rb_i\Gamma_i)$ in this case.\par
Finally, we consider the case that $D$ is a component of $\Gamma_i$ for at least one $1\leq i \leq r$ and that $f^{-1}(D)$ contains exactly one prime divisor $E$. In this case, it follows that 
\begin{align*}
    1-{\rm coeff}_D(B_Y)&={\rm lct}_{D}(X,B;f^*D)\\
    &=1-{\rm coeff}_E(B)\\
    &=1-\sum_{i=1}^rb_i{\rm coeff}_E(\Delta_i)\\
    &=1-\sum_{i=1}^rb_i{\rm coeff}_D(\Gamma_i),
\end{align*}
hence that ${\rm coeff}_D(B_Y)={\rm coeff}_D(\sum_{i=1}^rb_i\Gamma_i).$\par
We conclude that $B_Y=\sum_{i=1}^rb_i\Gamma_i.$ All other assertions in the statement of the theorem are now clear.
\end{proof}

\section{Generalized Bott towers}

\subsection{Decreasing relative dimension}
\begin{lemma}\label{lem:lin-equiv-tot-space-to-fiber}
Let $f\colon X \rightarrow Y$ be a fibration of $\qq$-factorial projective varieties, and let $(X,B)$ be a log Calabi--Yau pair of complexity zero. Denote by $m$ the number of components of $\lfloor B \rfloor$ which are horizontal over $Y.$ Then the general fiber $F$ of $f$ admits a log Calabi--Yau pair $(F,B_F)$ of complexity zero satisfying $\left\lvert \lfloor B_F \rfloor \right\rvert\leq m.$
\end{lemma}
\begin{proof}
Let $g\colon (\widetilde{X},\widetilde{B})\rightarrow (X,B)$ be a dlt modification. Choose a toric structure $(\widetilde{X},\Delta_{\widetilde{X}})$ associated to $(\widetilde{X},\widetilde{B}),$ and denote by $(X,\Delta_X)$ and $(Y,\Delta_Y)$ the toric structures induced by $(\widetilde{X},\Delta_{\widetilde{X}})$. Denote by $U=Y\setminus\Delta_Y$ the open orbit in $Y$ and by $\widetilde{F}$ the general fiber of $f\circ g.$ We have a commutative diagram
\[
\xymatrix{ 
\widetilde{F}\times U\ar[d]_-{p\times U} \ar@{^{(}->}[r] & \widetilde{X}\ar[d]^-{g} \\ 
F\times U \ar[d]_-{pr_2} \ar@{^{(}->}[r]& X \ar[d]^-{f}\\
U \ar@{^{(}->}[r]& Y
}
\]
in which:
\begin{enumerate}
    \item all squares are Cartesian,
    \item all morphisms are toric,
    \item $p\colon \widetilde{F}\rightarrow F$ is a projective birational toric morphism.
\end{enumerate}
Write $(\widetilde{F},\Delta_{\widetilde{F}})$ and $(F,\Delta_F)$ for the corresponding toric structures. \par
For each component $D$ of $\Delta_F,$ denote by $D_X$ the closure in $X$ of $D\times U.$ To obtain the desired result, it suffices by ~\cite[Corollary 2.33]{KM98} to show that the linear system $\lvert D \rvert$ on $F$ is positive-dimensional whenever $D_X$ appears in $B$ with coefficient less than one. So suppose $D_X$ appears in $B$ with coefficient less than one. Denote by $\widetilde{D}$ the strict transform on $\widetilde{F}$ of $D,$ and by $\widetilde{D}_{\widetilde{X}}$ the closure in $\widetilde{X}$ of $\widetilde{D}\times U.$ Noting that $\widetilde{D}_{\widetilde{X}}$ is the strict transform on $\widetilde{X}$ of $D_X,$ it follows that $\widetilde{D}_{\widetilde{X}}$ appears in $\widetilde{B}$ with coefficient less than one. By Lemma ~\ref{lem:conjugate-tori}, there is another component $\widetilde{D}'_{\widetilde{X}}$ of $\widetilde{B}$ such that $\widetilde{D}_{\widetilde{X}}\sim_\qq \widetilde{D}'_{\widetilde{X}}.$ Since $\widetilde{X}$ is smooth, it follows that the prime divisors $\widetilde{D}_{\widetilde{X}}$ and $\widetilde{D}'_{\widetilde{X}}$ are Cartier. Since $\widetilde{X}$ is a projective toric variety, it follows from ~\cite[Proposition 4.2.5]{CLS11} that its Picard group is torsion-free. Together, these imply the linear equivalence $\widetilde{D}_{\widetilde{X}}\sim \widetilde{D}'_{\widetilde{X}}.$ For general $y\in U,$ the Cartier divisor $\widetilde{D}'_{\widetilde{X}}|_{\widetilde{F}\times\{y\}}$ is an element of $\lvert\widetilde{D}\rvert$ different from $\widetilde{D}.$ Pushing forward to $F,$ we see that $\lvert D\rvert$ contains divisors different from $D.$
\end{proof}

\begin{lemma}\label{lem:round-down-and-sing-locus}
    Let $X$ be a $\qq$-factorial projective variety of Picard rank one. Suppose there exists an irreducible component $Z$ of ${\rm Sing}(X)$ and a log Calabi--Yau pair $(X,B)$ of complexity zero such that $Z\nsubset {\rm Supp}(\lfloor B \rfloor).$ Let $E$ be any exceptional divisor extracted by the normalized blow-up of $Z,$ and let $g\colon\widetilde{X}\rightarrow X$ be the extraction of $E.$ Then every $(-E)$-MMP terminates with a fibration to a positive-dimensional base.
\end{lemma}
\begin{proof}
Denote by $B_1,\hdots, B_k$ the irreducible components of $\lfloor B \rfloor.$ Choose any prime divisor $D\subset X$ containing $Z.$ Since $Z\subset D,$ there is a positive rational number $b$ such that $g^*D=\widetilde{D}+bE,$ where $\widetilde{D}$ dentoes the strict transform on $\widetilde{X}$ of $D.$ In contrast, $g^*B_i$ has irreducible support for each $1\leq i \leq k$ since $Z\nsubset B_i.$ As $X$ is $\qq$-factorial and of Picard rank one, there are positive rational numbers $a_1,\hdots, a_k$ such that $B_i\sim_\qq a_iD$ for each $1\leq i \leq k.$ We then have that $$g^*B_i\sim_\qq a_i(\widetilde{D}+bE)$$
for each $1\leq i \leq k.$\par
Since $Z$ is an irreducible component of ${\rm Sing}(X)$ and $(X,B)$ is a log Calabi--Yau pair of complexity zero, it follows from Corollary ~\ref{cor:lcps-from-sing} that $E$ is a log canonical place of $(X,B).$ Thus, the pair $(\widetilde{X},\widetilde{B})$ induced by log pullback is a log Calabi--Yau pair of complexity zero. Since $\widetilde{X}$ is a Mori dream space and $-E$ is not pseudo-effective, it follows that every $(-E)$-MMP terminates with a fibration. Let 
\[
\xymatrix{ 
\widetilde{X}=\widetilde{X}_0 \ar@{-->}[r]^-{f_1}
 & \widetilde{X}_1 \ar@{-->}[r]
 & \hdots \ar@{-->}[r]^-{f_m}
 & \widetilde{X}_m \ar[r]^-{h}& Y
}
\]
be such an MMP. Here, $f_1,\hdots, f_m$ are birational contractions with $f_1,\hdots, f_{m-1}$ small, and $h$ is an extremal fibration. Since $h$ is an extremal fibration, to show that $Y$ is positive-dimensional it suffices to show that the Picard rank of $\widetilde{X}_m$ is at least two. Since $\widetilde{X}_0$ is of Picard rank two and $f_1,\hdots, f_{m-1}$ are small, it follows that $\widetilde{X}_{m-1}$ is of Picard rank two. Thus, it suffices to show that $f_m$ is not a divisorial contraction.\par
Assume, for a contradiction, that $f_m$ is a divisorial contraction. For each $0\leq i \leq m,$ denote by $\widetilde{B}_{i,1},\hdots, \widetilde{B}_{i,k}, \widetilde{B}_i, \widetilde{D}_i$ and $E_i$ the pushforwards to $\widetilde{X}_i$ of $g^*B_1,\hdots, g^*B_k,\widetilde{B}, \widetilde{D}$ and $E$, respectively. Then $(\widetilde{X}_i,\widetilde{B}_i)$ is a log Calabi--Yau pair of complexity zero. It follows from Lemma \ref{lem:compl-zero-degn-divs} that $f_m$ cannot contract any divisor unequal to one of $\widetilde{B}_{m-1,1},\hdots,\widetilde{B}_{m-1,k}$ or $\widetilde{E}_{m-1}$. In particular, $f_m$ does not contract $\widetilde{D}_{m-1}.$ We also note that $f_m$ cannot contract $\widetilde{E}_{m-1},$ since $f_m$ must be $\widetilde{E}_{m-1}$-positive as a step of a $(-E)$-MMP. Thus, $f_m$ must contract $\widetilde{B}_{m-1,i}$ for some $1\leq i \leq k.$ But $\widetilde{B}_{m-1,i}\sim_\qq a_i(\widetilde{D}_{m-1}+bE_{m-1})$ on $\widetilde{X}_{m-1},$ from which it follows that $\widetilde{D}_m\sim_\qq -bE_m$ on $\widetilde{X}_m.$ But this is nonsense, as $\widetilde{X}_m$ is projective, $\widetilde{D}_m$ and $E_m$ are nonzero effective divisors, and $b>0.$ We conclude that $f_m$ cannot be a divisorial contraction, as desired.
\end{proof}

\begin{lemma}\label{lem:round-down-relative}
    Let $f\colon X\rightarrow Y$ be an extremal fibration of $\qq$-factorial projective toric varieties. Denote by $U\subset Y$ the open orbit and by $F$ the general fiber of $f,$ and fix a toric isomorphism $f^{-1}(U)\cong F\times U$ over $U.$ Assume there is an irreducible component $Z$ of ${\rm Sing}(F)$ satisfying the hypotheses of Lemma ~\ref{lem:round-down-and-sing-locus}. Let $E_F$ be any divisor extracted from $F$ by the normalized blow-up of $Z,$ and let $g \colon \widetilde{X}\rightarrow X$ be the extraction of a divisor $E$ corresponding to the valuation induced by $E_F\times U\subset f^{-1}(U).$ Then we may run a $(-E)$-MMP over $Y$ which terminates with a fibration to a base of positive relative dimension over $Y.$
\end{lemma}
\begin{proof}
Denote by $g_F\colon \widetilde{F}\rightarrow F$ the extraction of $E_F.$ Thus, we may identify the restriction of $g$ over $U$ with $g_F\times U\colon \widetilde{F}\times U \rightarrow F\times U.$ By Lemma ~\ref{lem:fiberwise-mmp}, given any $(-E_F\times U)$-MMP over $U$
\[
\xymatrix{ 
\widetilde{F}\times U = \widehat{F}_0 \ar@{-->}[r]^-{\widehat{f}_0}
 & \widehat{F}_1 \ar@{-->}[r]
 & \hdots \ar@{-->}[r]^-{\widehat{f}_{m-1}}
 & \widehat{F}_m \ar[r]^-{\widehat{h}}& \widehat{W},
}
\]
there is a $(-E_F)$-MMP
\[
\xymatrix{ 
\widetilde{F} = \widetilde{F}_0 \ar@{-->}[r]^-{\widetilde{f}_0}
 & \widetilde{F}_1 \ar@{-->}[r]
 & \hdots \ar@{-->}[r]^-{\widetilde{f}_{m-1}}
 & \widetilde{F}_m \ar[r]^-{\widetilde{h}}& \widetilde{W}
}
\]
satisfying
\begin{enumerate}
    \item $\widehat{F}_i\cong\widetilde{F}_i\times U$ for each $0\leq i \leq m,$
    \item $\widehat{W}\cong \widetilde{W}\times U,$
    \item the identifications above identify $\widehat{f}_i$ with $\widetilde{f}_i\times U$ for each $1\leq i \leq m-1$ and identify $\widehat{h}$ with $\widetilde{h}\times U.$ 
\end{enumerate} By Lemma ~\ref{lem:round-down-and-sing-locus}, the codomain $\widetilde{W}$ of the fibration $\widetilde{h}$ is positive-dimensional. It follows that $\widehat{W}$ has positive relative dimension over $U.$\par
Run a $(-E)$-MMP over $Y.$ Since $\widetilde{X}$ is a Mori dream space and since $-E$ is not pseudo-effective over $Y$ by Lemma ~\ref{lem:horiz-pseudoeff}, it follows that this MMP must terminate with a fibration over $Y.$ Denote by
\[
\xymatrix{ 
X = X_0 \ar@{-->}[r]^-{g_0}
 & X_1 \ar@{-->}[r]^-{g_1}
 & \hdots \ar@{-->}[r]^-{g_{m-1}}
 & X_m \ar[r]^-{g_{m}}
 & Z
}
\] the steps of this MMP, with $g_i$ birational for $0\leq i \leq m-1$ and $g_m$ a fibration. Denote by $V_i\subset X_i$ the preimage of $U$ in $X_i$ for each $0\leq i \leq m,$ by $W$ the preimage of $U$ in $Z,$ and by $h_i$ the restriction to $V_i$ of $g_i$ for each $0\leq i \leq m.$ The fibration $g_m$ cannot restrict to an isomorphism between $V_m$ and $W,$ so Lemma ~\ref{lem:rel-mmp-over-open} gives us indices $0\leq i_0 <\hdots < i_r=m$ such that $h_i$ is an isomorphism for $i\notin \{i_0,\hdots, i_{r}\}$ and
\[
\xymatrix{ 
V_0 = V_{i_0} \ar@{-->}[r]^-{h_{i_0}}
 & V_{i_1} \ar@{-->}[r]^-{h_{i_1}}
 & \hdots \ar@{-->}[r]^-{h_{i_{r-1}}}
 & V_{i_r} \ar[r]^-{h_{i_r}}
 & W
}
\]
are the steps of an $(-E)|_{V_0}=(-E_F\times U)$-MMP over $U.$ It follows from the arguments above that $W$ must be of positive relative dimension over $U.$ Thus, $Z$ must be of positive relative dimesnion over $Y.$
\end{proof}

\begin{lemma}\label{lem:fiber-switcharoo}
Let $f\colon X\rightarrow Y$ be an extremal fibration between $\qq$-factorial projective varieties. Let $(X,B)$ be a log Calabi--Yau pair of complexity zero such that $\lfloor B \rfloor$ contains exactly one component horizontal over $Y$. Assume that the general fiber $F$ of $f$ is singular but admits no log Calabi--Yau pair of complexity zero satisfying the hypotheses of Lemma ~\ref{lem:round-down-and-sing-locus}. Then $F\cong \pp(1,c_1,\hdots, c_n)$ with $c_i\geq 2$ for each $1\leq i \leq n.$ Moreover, there is a commutative diagram 
\[
\xymatrix{
(X,B)\ar[rd]_-{f} & 
 & (X',B')\ar[dl]^-{f'} \ar@{-->}[ll]_-{\phi}\\
& Y & 
}
\]
where $\phi$ is birational and extracts only log canonical places of $(X,B)$ such that one of the following holds:
\begin{enumerate}
    \item $X'$ admits a fibration over $Y$ to a base of positive relative dimension over $Y,$
    \item $f'$ is an extremal fibration whose general fiber $F'$ satisfies $F'\cong \pp(1,1,c'_2,\hdots, c'_n)$.
\end{enumerate}
\end{lemma}
\begin{proof}
Since $\lfloor B \rfloor$ has only one component which is horizontal over $Y,$ we can use Lemma ~\ref{lem:lin-equiv-tot-space-to-fiber} to obtain a log Calabi--Yau pair $(F,B_F)$ of complexity zero such that $\lfloor B_F\rfloor$ has at most one component. Since $F$ is singular and admits no log Calabi--Yau pair of complexity zero satisfying the hypotheses of Lemma ~\ref{lem:round-down-and-sing-locus}, it follows that $\lfloor B_F\rfloor$ must contain exactly one component and that ${\rm Sing}(F)\subset {\rm Supp}\left(\lfloor B_F\rfloor\right)$.\par 
Perform a dlt modification $g\colon (\widetilde{X}, \widetilde{B})\rightarrow (X,B).$ Choose a toric structure $(\widetilde{X},\Delta_{\widetilde{X}})$ associated to $(\widetilde{X},\widetilde{B}),$ and denote by $(X,\Delta_X)$ and $(Y,\Delta_Y)$ the toric structures induced by $(\widetilde{X},\Delta_{\widetilde{X}})$. Denote by $U=Y\setminus\Delta_Y$ the open orbit in $Y$ and by $\widetilde{F}$ the general fiber of $f\circ g.$ We have a commutative diagram
\[
\xymatrix{ 
\widetilde{F}\times U\ar[d]_-{p\times U} \ar@{^{(}->}[r] & \widetilde{X}\ar[d]^-{g} \\ 
F\times U \ar[d]_-{pr_2} \ar@{^{(}->}[r]& X \ar[d]^-{f}\\
U \ar@{^{(}->}[r]& Y
}
\]
in which:
\begin{enumerate}
    \item all squares are Cartesian,
    \item all morphisms are toric,
    \item $p\colon \widetilde{F}\rightarrow F$ is a projective birational toric morphism.
\end{enumerate}
Write $(\widetilde{F},\Delta_{\widetilde{F}})$ and $(F,\Delta_F)$ for the corresponding toric structures. Denote by $D_0,\hdots D_n$ the components of $\Delta_F,$ with $D_0=\lfloor B_F \rfloor,$ and denote by $u_0,\hdots,u_n\in N_F$ the primitive generators of the respective rays in the fan $\Sigma_F$ of $F.$ Write $\overline{D}$ for the closure in $X$ of $D_0\times U.$ \par
Since the torus-invariant point corresponding to ${\rm Cone}(u_1,\hdots,u_n)$ lies in the complement of ${\rm Supp}(D_0),$ it follows that this cone must be smooth. In other words, $u_1,\hdots,u_n$ form a $\zz$-basis for $N_F.$ Thus, we may write $u_0=\sum_{i=1}^nb_iu_i$ for some negative integers $b_1,\hdots, b_n.$ Writing $c_0=1$ and $c_i=-b_i$ for $1\leq i\leq n,$ we have ${\rm gcd}(c_0,\hdots, c_n)=1$ and $\sum_{i=0}^nc_iu_i=0.$ It follows that $F\cong \pp(c_0,\hdots, c_n)=\pp(1,c_1,\hdots, c_n).$ Assume, for a contradiction, that $c_i=1$ for some $1\leq i \leq n.$ By ~\cite[Proposition 4.3.3]{CLS11}, the vector space $\Gamma\left(F,\mathcal{O}_F(D_0)\right)$ can be identified with the vector space spanned by those characters $m\in M_F$ satisfying the inequalities
$$\langle m, u_0\rangle\geq -1$$
and $$\langle m, u_i\rangle\geq 0$$
for $1\leq i \leq n.$ The character $0$ clearly satisfies these inequalities, as does the character $e_i$ defined by $\langle e_i,u_j\rangle=\delta_{i,j}$ since 
    $$\langle e_i, u_0 \rangle = b_i=-1.$$
It follows that $\Gamma\left(F,\mathcal{O}_F(D_0)\right)$ has dimension at least two, hence that the linear system $|D_0|$ is positive-dimensional. Choosing a general element $D_0'\in |D|$ and setting $B'_F=B_F+\frac{1}{2}(D_0'-D_0),$ it follows from ~\cite[Corollary 2.33]{KM98} that we obtain a log Calabi--Yau pair $(F, B'_F)$ of complexity zero with $\lfloor B'_F \rfloor=0.$ This contradicts our assumptions on $F,$ so we conclude that $c_i\geq 2$ for all $1\leq i \leq n.$ \par
Since $\widetilde{X}$ is smooth, it follows that $\widetilde{F}$ must be smooth. It follows that ${\rm Cone}(u_0,u_2,\hdots, u_n),$ viewed as a cone in $(N_{\widetilde{F}})_\rr,$ is a union of smooth cones in $\Sigma_{\widetilde{F}}.$ Choose some cone $\sigma\in \Sigma_{\widetilde{F}}^{(n)}$ contained in ${\rm Cone}(u_0,u_2,\hdots, u_n)$ which is of the form ${\rm Cone}(\widetilde{u}_0,\widetilde{u}_2,\hdots, \widetilde{u}_n)$ for $\widetilde{u}_2,\hdots, \widetilde{u}_n\in {\rm Cone}(u_2,\hdots, u_n).$ Since $\sigma$ is a smooth cone, it follows that $\widetilde{u}_0,\widetilde{u}_2,\hdots, \widetilde{u}_n$ form a $\zz$-basis for $N_{\widetilde{F}}.$ Since $\widetilde{u}_2,\hdots, \widetilde{u}_n\in {\rm Cone}(u_2,\hdots, u_n),$ this implies that $\widetilde{u}_0=-u_1+\sum_{i=2}^nb'_iu_i$ for negative integers $b'_2,\hdots, b'_n.$ Let $\widetilde{D}_0\subset \widetilde{F}$ be the divisor corresponding to $\widetilde{u}_0,$ and let $\widehat{D}\subset \widetilde{X}$ be the closure of $\widetilde{D}_0\times U$ in $\widetilde{X}.$ Note that $\widehat{D}$ is a $g$-exceptional divisor and hence is a log canonical place of $(X,B).$ Let $E$ be the reduced sum of all $g$-exceptional divisors except $\widehat{D},$ and run an $E$-MMP over $X.$ This MMP terminates with a projective birational morphism $\widehat{g}\colon \widehat{X}\rightarrow X$ whose only exceptional divisor is $\widehat{D}.$ Write $(\widehat{X},\widehat{B})$ for the pair induced from $(X,B)$ by log pullback and write $\widehat{F}$ for the general fiber of $f\circ \widehat{g}$. The pair $(\widehat{X},\widehat{B})$ is a log Calabi--Yau pair of complexity zero, and the restriction of $\widehat{g}$ over $U$ can be identified with $\widehat{p}\times U\colon \widehat{F}\times U \rightarrow F\times U$ for a projective birational toric morphism $\widehat{p}\colon \widehat{F}\rightarrow F$ whose only exceptional divisor is $\widetilde{D}_0.$\par
Run a $(-\widehat{D})$-MMP over $Y.$ Since $\widehat{D}$ is horizontal over $Y$, it follows from Lemma ~\ref{lem:horiz-pseudoeff} that this MMP must terminate with an extremal fibration $f'\colon X'\rightarrow W$ over $Y.$ We are done if $W$ is of positive relative dimension over $Y,$ so assume that that $W$ is not of positive relative dimension over $Y.$ Since $\widehat{X}$ does not contain any divisors which are degenerate over $Y,$ it follows that $W=Y$ in this case. Thus, the birational contraction $\widehat{X}\dashrightarrow X'$ must contract a divisor. Since $(\widehat{X},\widehat{B})$ is a log Calabi--Yau pair of complexity zero, it follows from Corollary ~\ref{cor:complexity-zero-and-exceptional-divisors} that the divisor contracted by this map must be a component of $\lfloor \widehat{B}\rfloor.$ Since $\widehat{X}$ does not contain any divisors which are degenerate over $Y,$ the divisor contracted by this map must be horizontal over $Y.$ The only components of $\lfloor \widehat{B}\rfloor$ which are horizontal over $Y$ are $\widehat{D}$ and the strict transform on $\widehat{X}$ of $\overline{D}.$ Since $\widehat{D}$ cannot be contracted by a birational contraction in a $(-\widehat{D})$-MMP, it follows that $\overline{D}$ is contracted. By considering the restriction over $U$ of this relative MMP, we see that the general fiber $F'$ of $f'$ is the target of a birational contraction $\widehat{F}\dashrightarrow F'$ which contracts $D_0.$ It follows that $F'\cong \pp(1,1,c_2',\hdots, c_n',)$ where $c'_i=-b'_i$ for $2\leq i \leq n.$
\end{proof}

\begin{lemma}\label{lem:multiple-horiz-components}
Let $f\colon X \rightarrow Y$ be an extremal fibration between $\qq$-factorial projective varieties. Let $(X,B)$ be a log Calabi--Yau pair of complexity zero. Suppose that $\lfloor B \rfloor$ contains at least two components which are horizontal over $Y.$ Then there exists a commutative diagram 
\[
\xymatrix{ 
(X,B)\ar[d]_-{f} & (X',B')\ar[l]_-{\phi}\ar[d]^-{g} \\ 
Y & W\ar[l]^-{h}
}
\]
satisfying
\begin{enumerate}
    \item $\phi$ is a crepant projective birational morphism extracting only log canonical places of $(X,B),$
    \item both $g$ and $h$ are fibrations.
\end{enumerate}
\end{lemma}
\begin{proof}
    Write $\lfloor B \rfloor_h$ for the sum of the components of $\lfloor B \rfloor$ which are horizontal over $Y$. Since the morphism $f$ is of relative Picard rank one, it follows from our assumptions that the components of $\lfloor B \rfloor_h$ are linearly dependent modulo $\qq$-linear equivalence over $Y.$ Thus, we may write $D_1\sim D_2+f^*H,$ where $D_1,D_2$ are effective Cartier divisors supported on $\lfloor B \rfloor_h$, not both zero and sharing no common component, and where $H$ is a Cartier divisor on $Y.$ We note that, by Lemma ~\ref{lem:horiz-pseudoeff}, no effective Cartier divisor supported on $\lfloor B \rfloor_h$ can be linearly trivial over $Y.$ It follows that both of the divisors $D_1,D_2$ must be nonzero. \par
    Perform a dlt modification $g \colon (\widetilde{X},\widetilde{B})\rightarrow (X,B).$ It follows from Theorem ~\ref{introthm:dlt-log-smth} that $(\widetilde{X},\lfloor \widetilde{B}\rfloor)$ is log smooth. We may write $g^*D_i=\widetilde{D}_i+\widetilde{D}$ for $i=1,2,$ where $\widetilde{D}_1,\widetilde{D}_2$ are effective Cartier divisors sharing no common component. These divisors $\widetilde{D}_1,\widetilde{D}_2$ are supported on $\lfloor \widetilde{B}\rfloor$, and they satisfy $g_*\widetilde{D}_i=D_i$ for $i=1,2$ and $\widetilde{D}_1\sim \widetilde{D}_2+(fg)^*H.$ For each $i=1,2,$ write $\widetilde{D}_i=\sum_{j=1}^{k_i}m_{ij}\widetilde{D}_{ij},$ where $\widetilde{D}_{i1},\hdots,\widetilde{D}_{ik_i}$ are distinct prime divisors and $m_{i1},\hdots,m_{ik_i}$ are positive integers. Write $$N=\left\lvert\{(j_1,j_2)\in \{1,\hdots, k_1\}\times \{1,\hdots, k_2\}\vert \widetilde{D}_{1j_1}\cap \widetilde{D}_{2j_2}\neq\emptyset\}\right\rvert.$$ 
    Denote by $M$ the largest integer such that there is a nonempty intersection $\widetilde{D}_{1j_1}\cap\widetilde{D}_{2j_2}$ with $M\in \{m_{1j_1},m_{2j_2}\},$ setting $M=0$ in the event that $N=0.$ Denote by $R$ the number of times that $M$ appears as the coefficient of a divisor $\widetilde{D}_{ij_i}$ participating in some nonempty intersection $\widetilde{D}_{1j_1}\cap\widetilde{D}_{2j_2},$ setting $R=0$ in the event that $N=0.$ If $N=0,$ then the divisors $\widetilde{D}_1,\widetilde{D}_2$ have disjoint support. If $N>0,$ choose some nonempty intersection $\widetilde{D}_{1j_1}\cap\widetilde{D}_{2j_2}$ such that $M\in \{m_{1j_1},m_{2j_2}\}$ and write $m=\min\{m_{1j_1},m_{2j_2}\}.$ If possible, make this choice so that $m=M.$ Let $h\colon \widehat{X}\rightarrow \widetilde{X}$ be the blow up of $\widetilde{X}$ along $\widetilde{D}_{1j_1}\cap\widetilde{D}_{2j_2},$ and denote by $(\widehat{X},\widehat{B})$ the log pullback of $(\widetilde{X},\widetilde{B}).$ There is a unique $h$-exceptional divisor $E,$ which is a log canonical place of $(\widetilde{X},\widetilde{B}),$ and the pair $(\widehat{X},\widehat{B})$ is a dlt log Calabi--Yau pair of complexity zero. Write $\widehat{D}_i=h^*\widetilde{D}_i-mE.$ Then $\widehat{D}_1,\widehat{D}_2$ are effective Cartier divisors supported on $\lfloor \widehat{B} \rfloor$, sharing no common component, which satisfy $(gh)_*\widehat{D}_i=D_i$ for $i=1,2$ and $\widehat{D}_1\sim\widehat{D}_2+(fgh)^*H.$ Defining $\widehat{N},\widehat{M},\widehat{R}$ for $\widehat{D}_1,\widehat{D}_2$ in the same way we defined $N,M,R,$ respectively, for $\widetilde{D}_1,\widetilde{D}_2$ above, we note that we have $(\widehat{M},\widehat{R},\widehat{N})<(M,R,N)$ when lexicographically ordered. Replacing $(\widetilde{X},\widetilde{B}), \widetilde{D}_1,\widetilde{D}_2$ with $(\widehat{X},\widehat{B}),\widehat{D}_1,\widehat{D}_2$ and continuing in this manner, we see that we may assume that $\widetilde{D}_1$ and $\widetilde{D}_2$ have disjoint support.\par
    Using the fact that $\widetilde{D}_1$ and $\widetilde{D}_2$ have disjoint support and satisfy $\widetilde{D}_1\sim \widetilde{D}_2+(fg)^*H,$ we obtain a morphism $\widetilde{X}\rightarrow\pp(\mathcal{O}_Y\oplus \mathcal{O}_Y(H))$ over $Y.$ The divisors $\widetilde{D}_1,\widetilde{D}_2$ map to disjoint sections of $\pp(\mathcal{O}_Y\oplus \mathcal{O}_Y(H))\rightarrow Y.$ Since both $\widetilde{D}_1$ and $\widetilde{D}_2$ dominate $Y,$ it follows that the morphism $\widetilde{X}\rightarrow\pp(\mathcal{O}_Y\oplus \mathcal{O}_Y(H))$ must be surjective. Stein factorization then gives us the desired fibrations $g$ and $h.$
\end{proof}

\begin{lemma}\label{lem:decrease-rel-dimn}
Let $f\colon X \rightarrow Y$ be an extremal fibration between $\qq$-factorial projective varieties. Let $(X,B)$ be a log Calabi--Yau pair of complexity zero. Suppose that the fibers of $f$ are singular. Then there exists a commutative diagram 
\[
\xymatrix{ 
(X,B)\ar[d]_-{f} & (X',B')\ar@{-->}[l]_-{\phi}\ar[d]^-{g} \\ 
Y & W\ar[l]^-{h}
}
\]
satisfying
\begin{enumerate}
    \item $\phi$ is a crepant birational map extracting only log canonical places of $(X,B),$
    \item both $g$ and $h$ are fibrations.
\end{enumerate}
\end{lemma}
\begin{proof}
The desired result follows immediately from Lemma ~\ref{lem:multiple-horiz-components} in the case that $\lfloor B\rfloor$ has at least two components which are horizontal over $Y.$ In what follows, we will assume that $\lfloor B\rfloor$ contains at most one component which is horizontal over $Y.$\par
By Lemma ~\ref{lem:lin-equiv-tot-space-to-fiber}, the general fiber $F$ of $f$ admits a log Calabi--Yau pair $(F,B_F)$ of complexity zero with $\lfloor B_F \rfloor \leq 1.$ We begin by considering the case in which this pair $(F,B_F)$ satisfies the hypotheses of Lemma ~\ref{lem:round-down-and-sing-locus}. It follows from Lemma ~\ref{lem:round-down-relative} that there exists a commutative diagram 
\[
\xymatrix{ 
(X,B)\ar[d]_-{f} & (\widetilde{X},\widetilde{B}) \ar[l]_-{\psi} \ar@{-->}[r]^-{\psi'} & (X',B')\ar[d]^-{g} \\ 
Y & & W\ar[ll]^-{h}
}
\]
satisfying
\begin{enumerate}
    \item[(1.1)] $\psi$ is a crepant birational morphism extracting a single divisor $E$, which is a log canonical place of $(X,B)$,
    \item[(1.2)] $\psi'$ is a birational contraction obtained from running a $(-E)$-MMP over $Y,$
    \item[(1.3)] both $g$ and $h$ are fibrations.
\end{enumerate}
Setting $\phi=\psi\circ (\psi')^{-1},$ we obtain the desired result in this case.\par
From now on, we assume that $(F,B_F)$ does not satisfy the hypotheses of Lemma ~\ref{lem:round-down-and-sing-locus}. We may apply Lemma ~\ref{lem:fiber-switcharoo} in this case to obtain a commutative diagram
\[
\xymatrix{
(X,B)\ar[rd]_-{f} & 
 & (X',B')\ar[dl]^-{f'} \ar@{-->}[ll]_-{\psi}\\
& Y & 
}
\]
where $\phi$ is birational and extracts only log canonical places of $(X,B)$ and one of the following holds:
\begin{enumerate}
    \item[(2.1)] $X'$ admits a fibration over $Y$ to a base of positive relative dimension over $Y,$
    \item[(2.2)] $f'$ is an extremal fibration whose fiber $F'$ satisfies $F'\cong \pp(1,1,c'_2,\hdots, c'_n)$ with $1\leq c'_2\leq\hdots\leq c'_n.$
\end{enumerate}
We are done in the event that (1) holds, so assume that (2) holds. In this case, it follows from Lemma ~\ref{lem:fiber-switcharoo} that $\lfloor B' \rfloor$ has at least two components which are horizontal over $Y$ or that $F'$ admits a log Calabi--Yau pair of complexity zero satisfying the hypotheses of Lemma ~\ref{lem:round-down-and-sing-locus}. By the arguments of the previous paragraphs, we obtain a commutative diagram 
\[
\xymatrix{ 
(X',B')\ar[d]_-{f'} & (X'',B'')\ar@{-->}[l]_-{\psi'}\ar[d]^-{f''} \\ 
Y & W\ar[l]^-{h}
}
\]
satisfying
\begin{enumerate}
    \item[(3.1)] $\phi$ is a crepant birational map extracting only log canonical places of $(X',B'),$
    \item[(3.2)] both $g$ and $h$ are fibrations.
\end{enumerate}
Setting $\phi=\psi\circ \psi',$ we obtain the desired result in this case.
\end{proof}

\subsection{Projective space bundles}

\begin{lemma}\label{lem:obtain-bundle}
Let $f\colon X \rightarrow Y$ be an extremal fibration between $\qq$-factorial projective varieties. Let $(X,B)$ be a log Calabi--Yau pair of complexity zero, and suppose that $Y$ is smooth.
Then there exists a commutative diagram 
\[
\xymatrix{
(X,B)\ar[rd]_-{f} & 
 & (X',B')\ar[dl]^-{f'} \ar@{-->}[ll]_-{\phi}\\
& Y & 
}
\]
satisfying
\begin{enumerate}
    \item $\phi$ is a crepant birational map extracting only log canonical places of $(X,B),$
    \item $f'$ is an extremal contraction which is a locally trivial fiber bundle.
\end{enumerate}
\end{lemma}
\begin{proof}
Perform a dlt modification $g\colon (\widetilde{X},\widetilde{B})\rightarrow(X,B).$ Choose a toric structure $(\widetilde{X},\Delta_{\widetilde{X}})$ associated to $(\widetilde{X},\widetilde{B}),$ and let $(X,\Delta)$ and $(Y,\Delta_Y)$ be the toric structures induced on $X$ and $Y,$ respectively. Write $N_{\widetilde{X}}, N_X$ and $N_Y$ for the cocharacter lattices and $\Sigma_{\widetilde{X}}, \Sigma_X$ and $\Sigma_Y$ for the fans corresponding the toric structures $(\widetilde{X},\Delta_{\widetilde{X}}), (X,\Delta)$ and $(Y,\Delta_Y),$ respectively. It follows from Lemma ~\ref{lem:splitting-fans} that the fan $\Sigma_X$ can be expressed as a sum $$\Sigma_X=\Sigma_{X,F}+\Sigma_{X,Y}$$
    of subfans $\Sigma_{X,F},\Sigma_{X,Y}\subset \Sigma_X$, where 
    \begin{enumerate}
        \item[(1.1)] $\Sigma_{X,F}$ has support equal to ${\rm Ker}(f_*),$
        \item[(1.2)] $f_*$ restricts to a bijection $\tau\xrightarrow{\cong}f_*(\tau)$ for each $\tau\in \Sigma_{X,Y},$ 
        \item[(1.3)] the assignment $\tau\mapsto f_*(\tau)$ determines a bijection $\Sigma_{X,Y}\xrightarrow{\cong}\Sigma_Y.$
    \end{enumerate}\par
We begin by assuming that, for each $\widetilde{\sigma}\in \Sigma_{X,Y}^{(1)},$ the primitive generator of $\widetilde{\sigma}$ is mapped by $f_*$ to the primitive generator of $f_*(\widetilde{\sigma}).$ We will show that, in this special case, $f$ is already a locally trivial fiber bundle. To show this, it suffices by ~\cite[Theorem 3.3.19]{CLS11} to show that $f_*(\widetilde{\sigma}\cap N_X)=f_*(\widetilde{\sigma})\cap N_Y$ for each $\widetilde{\sigma}\in \Sigma_{X,Y}.$ Write $\widetilde{v}_1,\hdots, \widetilde{v}_k\in N_X$ for the primitive generators of $1$-dimensional faces of $\widetilde{\sigma},$ and write $v_i=f_*(\widetilde{v}_i)$ for each $1\leq i \leq k.$
It follows from our assumption that $v_1,\hdots, v_k$ are the primitive generators of the $1$-dimensional faces of $\sigma=f_*(\widetilde{\sigma}).$ Given $w\in \sigma\cap N_Y,$ it follows from the smoothness of $Y$ that $w=\sum_{i=1}^ka_iv_i$ for some nonnegative integers $a_1,\hdots, a_k.$ We obtain a lattice vector $\widetilde{w}=\sum_{i=1}^ka_i\widetilde{v}_i\in N_X$ satisfying $f_*(\widetilde{w})=w.$ It follows that $f$ is a locally trivial fiber bundle, as claimed.\par
We now turn to treat the general case. Choose a cone $\tau\in \Sigma_{X,F}$ of dimension $r=\dim {\rm Ker}(f_*)$. Given a ray $\widetilde{\sigma}\in \Sigma_{X,Y}^{(1)},$ consider the cone $\tau+\widetilde{\sigma}$ viewed as a cone in $(N_{\widetilde{X}})_\rr.$ Since $g$ is a projective birational morphism, this cone is a union of cones in $\Sigma_{\widetilde{X}}.$ Since $(\widetilde{X},\widetilde{B})$ is dlt, it follows from Theorem ~\ref{introthm:dlt-log-smth} that $\widetilde{X}$ is smooth. Thus, $\tau+\widetilde{\sigma}$ is a union of smooth cones in $\Sigma_{\widetilde{X}}.$ Choose a cone $\gamma\in \Sigma_{\widetilde{X}}^{(r+1)}$ such that $\gamma\subset \tau+\widetilde{\sigma}$ and such that one of the $r$-dimensional faces of $\gamma$ is contained in ${\rm Ker}(f_*\circ g_*).$ Write $\widehat{\tau}_1,\hdots, \widehat{\tau}_r$ for the $1$-dimensional faces of $\gamma$ contained in ${\rm Ker}(f_*\circ g_*)$ and $\widehat{\sigma}$ for the remaining $1$-dimensional face of $\gamma.$ We claim that the primitive generator of $\widehat{\sigma}$ is mapped by $f_*\circ g_*$ to the primitive generator of $\sigma=f_*(\widetilde{\sigma})=(f_*\circ g_*)(\widehat{\sigma}).$ To see this, write $\widehat{u}_1,\hdots, \widehat{u}_r,\widehat{v}$ for the primitive generators of $\widehat{\tau}_1,\hdots,\widehat{\tau}_r,\widehat{\sigma},$ respectively. Since the cone $\gamma$ generated by the rays $\widehat{\tau}_1,\hdots,\widehat{\tau}_r,\widehat{\sigma}$ is smooth, it follows that we may extend $\widehat{u}_1,\hdots, \widehat{u}_r,\widehat{v}$ to a $\zz$-basis for $N_{\widetilde{X}},$ say $\widehat{u}_1,\hdots, \widehat{u}_r,\widehat{v},\widehat{w}_1,\hdots,\widehat{w}_{m-1}.$ Since $f_*\circ g_*\colon N_{\widetilde{X}}\rightarrow N_Y$ is surjective, $\widehat{u}_1,\hdots, \widehat{u}_r\in {\rm Ker}(f_*\circ g_*)$ and $m=\dim Y,$ it follows that $(f_*\circ g_*)(\widehat{v}),(f_*\circ g_*)(\widehat{w}_1),\hdots,(f_*\circ g_*)(\widehat{w}_{m-1})$ is a $\zz$-basis for $N_Y.$ This implies that $(f_*\circ g_*)(\widehat{v}),$ which is a positive integer multiple of the primitive generator of $\sigma,$ must be equal to the primitive generator of $\sigma.$\par
Denote by $E$ the reduced sum of all $g$-exceptional divisors on $\widetilde{X}.$ Denote by $\widetilde{\sigma}_1,\hdots,\widetilde{\sigma}_k\in \Sigma^{(1)}_{X,Y}$ those rays whose primitive generator is not mapped by $f_*$ to the primitive generator of $f_*(\widetilde{\sigma}).$ For each $1\leq i \leq k,$ use the arguments of the previous paragraph to choose a $g$-exceptional divisor $E_i$ whose primitive generator is mapped by $f_*\circ g_*$ to the primitive generator of $f_*(\widetilde{\sigma}_i).$ Run a $(E-\sum_{i=1}^kE_i)$-MMP over $X.$ This MMP terminates with a projective birational morphism $h\colon \overline{X}\rightarrow X$. The map $\widetilde{X}\dashrightarrow \overline{X}$ contracts each component of $E-\sum_{i=1}^kE_i$ but contracts none of the divisors $E_1,\dots, E_k.$ For each $1\leq i \leq k,$ write $\widetilde{E}_i$ for the divisor on $\overline{X}$ corresponding to the ray $\widetilde{\sigma}_i.$ Note, for each $1\leq i \leq k,$ that $(f\circ h)(E_i)=(f\circ h)(\widetilde{E}_i)$ and hence that the divisor $\widetilde{E}_i$ is degenerate over $Y$. Run a $(\sum_{i=1}^k\widetilde{E}_i)$-MMP over $Y.$ By Lemma ~\ref{lem:degn-divs-neg} and the fact that $\overline{X}$ is a Mori dream space, it follows that this MMP terminates after contracting the divisors $\widetilde{E}_1,\hdots, \widetilde{E}_k$ and no others. We obtain a birational map $\overline{X}\dashrightarrow X'$ and an extremal fibration $f'\colon X'\rightarrow Y$.\par Denote by $\phi\colon X' \dashrightarrow X$ the inverse of the composite 
\[
\xymatrix{ 
X \ar@{-->}[r]^-{g^{-1}}
 & \widetilde{X} \ar@{-->}[r]
 & \overline{X} \ar@{-->}[r]
 & X',
}
\]
and denote by $B'$ the pushforward to $X'$ of $\widetilde{B}.$ Then $(X',B')$ is a log Calabi--Yau pair of complexity zero, and $\phi\colon (X',B')\dashrightarrow (X,B)$ is a crepant birational map extracting only log canonical places of $(X,B)$. We claim that the morphism $f'$ is a locally trivial fiber bundle. Let $(X',\Delta_{X'})$ be the toric structure induced from $(\widetilde{X},\Delta_{\widetilde{X}})$ by the contraction $\widetilde{X}\dashrightarrow X'.$ It follows from Lemma ~\ref{lem:splitting-fans} that the corresponding fan $\Sigma_{X'}$ can be expressed as a sum $$\Sigma_{X'}=\Sigma_{X',F'}+\Sigma_{X',Y}$$
    of subfans $\Sigma_{X',F'},\Sigma_{X',Y}\subset \Sigma_{X'}$, where 
    \begin{enumerate}
        \item[(2.1)] $\Sigma_{X,F}$ has support equal to ${\rm Ker}(f'_*),$
        \item[(2.2)] $f'_*$ restricts to a bijection $\tau\xrightarrow{\cong}f'_*(\tau)$ for each $\tau\in \Sigma_{X',Y},$ 
        \item[(2.3)] the assignment $\tau\mapsto f'_*(\tau)$ determines a bijection $\Sigma_{X',Y}\xrightarrow{\cong}\Sigma_Y.$
    \end{enumerate}\par
It follows from the construction of $X'$ that, for each ray $\widetilde{\sigma}\in \Sigma_{X',Y}^{(1)},$ the primitive generator of $\widetilde{\sigma}$ is mapped by $f'_*$ to the primitive generator of $f'_*(\widetilde{\sigma}).$ It follows from the special case treated above that $f'$ is a locally trivial fiber bundle, as claimed.
\end{proof}

\begin{definition}\label{defn:rel-gen-Bott-tower}{\em
We say that a morphism $f\colon X \rightarrow Y$ is a \textit{relative generalized Bott tower} if it can be factored as $$X=X_0\xrightarrow{f_0}X_1\rightarrow\hdots\rightarrow X_{m}\xrightarrow{f_{m}}X_{m+1}=Y,$$
where each $f_i\colon X_i\rightarrow X_{i+1}$ is the projective space bundle associated to a direct sum of line bundles on $X_{i+1}.$
}
\end{definition}

\begin{theorem}\label{thm:proj-bundles-over-smth-base}
Let $f\colon X \rightarrow Y$ be a fibration between $\qq$-factorial projective varieties. Let $(X,B)$ be a log Calabi--Yau pair of complexity zero, and suppose that $Y$ is smooth.
Then there exists a commutative diagram 
\[
\xymatrix{
(X,B)\ar[rd]_-{f} & 
 & (X',B')\ar[dl]^-{f'} \ar@{-->}[ll]_-{\phi}\\
& Y & 
}
\]
satisfying
\begin{enumerate}
    \item $\phi$ is a crepant birational map extracting only log canonical places of $(X,B),$
    \item $f'$ is a relative generalized Bott tower.
\end{enumerate}
\end{theorem}
\begin{proof}
We begin by establishing the result in the special case that there exists no commutative diagram 
\[
\xymatrix{ 
(X,B)\ar[d]_-{f} & (Z,B_Z)\ar[d]^-{g} \ar@{-->}[l]_-{p} \\ 
Y & W\ar[l]^-{h}
}
\]
satisfying
\begin{enumerate}
    \item[(1.1)] $p$ is a crepant birational map extracting only log canonical places of $(X,B),$
    \item[(1.2)] $g$ and $h$ are both fibrations.
\end{enumerate}
We run a $K_X$-MMP over $Y,$ obtaining a commutative diagram 
\[
\xymatrix{ 
(X,B)\ar[d]_-{f} \ar@{-->}[r]^-{p} & (Z,B_Z)\ar[d]^-{g} \\ 
Y & W\ar[l]^-{h}
}
\]
where
\begin{enumerate}
    \item[(2.1)] $p$ is a crepant birational contraction map,
    \item[(2.2)] $g$ is a fibration.
\end{enumerate}
Since $$p^{-1}\colon (Z,B_Z)\dashrightarrow (X,B)$$ is a crepant birational map extracting only log canonical places of $(X,B),$ it follows from our assumption in this special case that the morphism $h$ must be birational. Denote by $(W,B_W)$ the log Calabi--Yau pair induced by $(Z,B_Z)$ via the canonical bundle formula, and write $B_Y=h_*B_W.$ By Lemma ~\ref{lem:follow-fibr}, there is a commutative diagram 
\[
\xymatrix{ 
(Z,B_Z)\ar[d]_-{g} & (Z',B_{Z'})\ar@{-->}[l]_-{\phi'}\ar[d]^-{g'} \\ 
(W,B_W) &
(Y,B_{Y})\ar@{-->}[l]^-{h^{-1}}
}
\]
where 
\begin{enumerate}
    \item[(3.1)] $\phi'$ is a crepant birational map extracting only log canonical places of $(Z,B_Z),$
    \item[(3.2)] $g'$ is an extremal fibration.
\end{enumerate}
By Lemma ~\ref{lem:obtain-bundle}, there exists a commutative diagram 
\[
\xymatrix{
(Z',B_{Z'})\ar[rd]_-{g'} & 
 & (Z'',B_{Z''})\ar[dl]^-{g''} \ar@{-->}[ll]_-{\phi''}\\
& Y & 
}
\]
satisfying
\begin{enumerate}
    \item[(4.1)] $\phi''$ is a crepant birational map extracting only log canonical places of $(Z',B_{Z'}),$
    \item[(4.2)] $g''$ is an extremal contraction which is a locally trivial fiber bundle.
\end{enumerate}
It follows from our assumption in this special case that $g''$ must have smooth fibers. Since $g''$ has relative Picard rank one, this implies that the fibers of $g''$ are isomorphic to $\pp^n$ for $n=\dim Z'' - \dim Y.$ It now follows from Lemma ~\ref{lem:proj-bundle-sum-line-bundles} that $g''$ is a relative generalized Bott tower.\par
For the general case, we induct on the relative dimension of $f.$ The base case $\dim X=\dim Y +1$ follows immediately from the special case established above. For the inductive step, we may assume that there exists a commutative diagram 
\[
\xymatrix{ 
(X,B)\ar[d]_-{f} & (Z,B_Z)\ar@{-->}[l]_-{p}\ar[d]^-{g} \\ 
Y & W\ar[l]^-{h}
}
\]
where
\begin{enumerate}
    \item[(5.1)] $p$ is a crepant birational map extracting only log canonical places of $(X,B),$
    \item[(5.2)] $g$ and $h$ are both fibrations.
\end{enumerate}
Let $(W,B_W)$ be the pair induced by $(Z,B_Z)$ via the canonical bundle formula. Then $(W,B_W)$ is a log Calabi--Yau pair of complexity zero by Theorem ~\ref{introthm:canonical-bundle-formula}, and it follows from the inductive hypothesis that there is a commutative diagram 
\[
\xymatrix{
(W,B_W)\ar[rd]_-{h} & 
 & (W',B_{W'})\ar[dl]^-{h'} \ar@{-->}[ll]_-{\psi}\\
& Y & 
}
\]
satisfying
\begin{enumerate}
    \item[(6.1)] $\psi$ is a crepant birational map extracting only log canonical places of $(W,B_W),$
    \item[(6.2)] $h'$ is a relative generalized Bott tower.
\end{enumerate}
In particular, $W'$ is smooth. By Lemma ~\ref{lem:follow-fibr}, there is a commutative diagram 
\[
\xymatrix{ 
(Z,B_Z)\ar[d]_-{g} & (Z',B_{Z'})\ar@{-->}[l]_-{\phi'}\ar[d]^-{g'} \\ 
(W,B_W) &
(W',B_{W'})\ar@{-->}[l]^-{\psi}
}
\]
where 
\begin{enumerate}
    \item[(7.1)] $\phi'$ is a crepant birational map extracting only log canonical places of $(Z,B_Z),$
    \item[(7.2)] $g'$ is a fibration.
\end{enumerate}
By the inductive hypothesis, there is a commutative diagram 
\[
\xymatrix{
(Z',B_{Z'})\ar[rd]_-{g'} & 
 & (Z'',B_{Z''})\ar[dl]^-{g''} \ar@{-->}[ll]_-{\phi''}\\
& W' & 
}
\]
satisfying
\begin{enumerate}
    \item[(8.1)] $\phi''$ is a crepant birational map extracting only log canonical places of $(Z',B_{Z'}),$
    \item[(8.2)] $g''$ is a relative generalized Bott tower.
\end{enumerate}
We obtain a commutative diagram 
\[
\xymatrix{
(X,B)\ar[rd]_-{f} & 
 & (Z'',B_{Z''})\ar[dl]^-{h'\circ g''} \ar@{-->}[ll]_-{p\circ \phi'\circ \phi''}\\
& Y. & 
}
\]
By construction, the composite $p\circ \phi'\circ \phi''$ is a crepant birational map extracting only log canonical places of $(X,B)$ and $h'\circ g''$ is a relative generalized Bott tower.
\end{proof}

\begin{proof}[Proof of Theorem ~\ref{introthm:gen-bott-towers}]
    Applying ~\cite[Lemma 2.31]{MM24} if necessary, we may assume that $(X,B)$ is a log Calabi--Yau pair of complexity zero and that $X$ is $\qq$-factorial. Now apply Theorem ~\ref{thm:proj-bundles-over-smth-base} to the case $Y=\Spec\kk.$
\end{proof}

\section{Examples and questions}




Let $L_1,L_2,L_3, L_4\subset \pp^2$ be distinct lines passing through a common point $p\in \pp^2,$ and let $L_5,L_6\subset \pp^2$ be two general lines. Write $B=\frac{1}{2}\sum_{i=1}^6L_i.$ Then $(\pp^2,B)$ is a log Calabi--Yau pair of complexity zero. We will make use of this pair in Examples ~\ref{ex:lambda_2<lambda_1}-~\ref{ex:lcc-not-toric-stratum-for-all}.

\begin{example}\label{ex:lambda_2<lambda_1}{\em
    In this example, we show that $\lambda_1$ and $\lambda_2$ need not be equal. Consider the log Calabi--Yau pair $(\pp^2,B)$ as above, and consider $\Delta=L_4+L_5+L_6.$ Then $\Delta\in \mathcal{A}(\pp^2,B).$ It is clear that $\lambda_1(\pp^2,B;\Delta)=\frac{1}{2}.$ We claim, however, that $\lambda_2(\pp^2,B;\Delta)=0.$ Indeed, given $\lambda\in \left[0,\frac{1}{2}\right]$ we have that the multiplicity of $\frac{B-\lambda\Delta}{1-\lambda}$ at the point $p$ is $1+\frac{1}{1-\lambda}.$ This multiplicity cannot exceed $2$ if $\left(\pp^2, \frac{B-\lambda\Delta}{1-\lambda}\right)$ is to be log canonical, so it follows that we must have $\lambda=0.$ 
    }
\end{example}

\begin{example}\label{ex:assoc-but-no-blend}{\em
    In this example, we show that not all associated pairs can participate in toric boundary arrangements. Consider the log Calabi--Yau pair $(\pp^2,B)$ as above, and consider $\Delta=L_4+L_5+L_6.$ Then $\Delta\in \mathcal{A}(\pp^2,B).$ Suppose that $B=\sum_{i=1}^rb_i\Delta_i,$ where $b_1,\hdots, b_r\in [0,1]$ are rational numbers with sum $\sum_{i=1}^rb_i=1,$ $\Delta_1,\hdots, \Delta_r\in \mathcal{A}(\pp^2, B)$ are distinct, and $\Delta_1=\Delta.$ Note that we must have $b_1\leq \frac{1}{2}.$ It follows that $$B'=\frac{B-b_1\Delta}{1-b_1}=\sum_{i=2}^r\frac{b_i}{1-b_1}\Delta_i$$
    is such that $(\pp^2,B')$ is log canonical. The computations of Example ~\ref{ex:lambda_2<lambda_1} show that we must have $b_1=0.$
    }
\end{example}

\begin{example}\label{ex:lcc-not-toric-stratum-for-all}
{\em
In this example, we show that a log canonical center of a log Calabi--Yau pair of complexity zero need not be a toric stratum of each associated divisor. Consider the log Calabi--Yau pair $(\pp^2,B)$ as above, and consider $\Delta=L_4+L_5+L_6.$ Then $\Delta\in \mathcal{A}(\pp^2,B).$ On the one hand, the point $p$ is a log canonical center of $(\pp^2, B)$ since the divisor $B$ has multiplicity $2$ at $p.$ On the other hand, the point $p$ is not a log canonical center of $(\pp^2,\Delta)$ since $\Delta$ only has multiplicity $1$ at $p.$ In particular, $p$ cannot be a toric stratum of $(\pp^2,\Delta).$
}
\end{example}

Fix a dimension $n\geq 2$ and an index $m\geq 3.$ Mauri and Moraga show in ~\cite[Theorem 1.6]{MM24} that every $n$-dimensional log Calabi--Yau pair of index one and birational complexity zero is crepant to a log Calabi--Yau of complexity zero supported on $\pp^n.$ In the following example, we show that there is no finite set $X_1,\hdots, X_r$ of $n$-folds with the property that if $(X,B)$ is a log Calabi--Yau $n$-fold of index $m$ and birational complexity zero then there is an index $1\leq i \leq r$ and a log Calabi--Yau pair $(X_i,B_{X_i})$ of complexity zero crepant to $(X,B).$

\begin{example}\label{ex:inf-many-models}{\em 
Fix integers $d\geq 1,n\geq 2$ and $m\geq 3.$ Denote by $X=\pp\left(\mathcal{O}_{\pp^{n-1}}\oplus\mathcal{O}_{\pp^{n-1}}(d)\right)$ and by $\pi\colon X \rightarrow \pp^{n-1}$ the projection. Denote by $S$ the section corresponding to the surjection $\mathcal{O}_{\pp^{n-1}}\oplus \mathcal{O}_{\pp^{n-1}}(d)\rightarrow \mathcal{O}_{\pp^{n-1}}$  and by $T$ the section corresponding to $\mathcal{O}_{\pp^{n-1}}\oplus \mathcal{O}_{\pp^{n-1}}(d)\rightarrow \mathcal{O}_{\pp^{n-1}}(d)$. Then $S$ has normal bundle $\mathcal{O}_{\pp^{n-1}}(-d)$ and $T$ has normal bundle $\mathcal{O}_{\pp^{n-1}}(d).$ Since $T$ is a nef divisor on the smooth toric variety $X$, it follows that the linear system $|T|$ is basepoint-free.\par Let $H$ be the reduced sum of $mn$ general hyperplanes in $\pp^{n-1},$ and let $T_1,\hdots, T_m\in |T|$ be $m$ general sections. Define $B=S+\frac{1}{m}(\sum_{i=1}^mT_i+\pi^*H).$ Then $B$ is a divisor whose support has simple normal crossings, and $(X,B)$ is a log Calabi--Yau pair of index $m$ and complexity zero. It follows from ~\cite[Corollary 2.31]{KM98} that the pair $(X,B)$ is terminal away from $S.$ \par
Suppose $\phi\colon (X,B)\dashrightarrow (Y,B_Y)$ is a crepant birational map, with $(Y,B_Y)$ another log Calabi--Yau pair of complexity zero. There exists a commutative diagram 
\[
\xymatrix{
\widetilde{X}\ar[d]_-{f} \ar@{-->}[r]^-{\psi} & \widetilde{Y}\ar[d]^-{g} \\
X \ar@{-->}[r]^-{\phi}& Y
}
\]
satisfying:
\begin{enumerate}
    \item $\widetilde{X}$ and $\widetilde{Y}$ are $\qq$-factorial,
    \item $\psi$ is an isomorphism in codimension one,
    \item $f$ is a projective birational morphism extracting only $\phi^{-1}$-exceptional divisors,
    \item $g$ is a projective birational morphism extracting only $\phi$-exceptional divisors.
\end{enumerate}
It follows from the crepancy of $\phi$ and the effectivity of $B_Y$ that every divisor extracted by $f$ has nonpositive discrepancy with respect to $(X,B).$ Every such divisor must have its center on $X$ contained in $S.$ The exceptional locus of $f$ is purely divisorial since $X$ is smooth, so it follows that $f$ is an isomorphism above $X\setminus S.$\par
Since $\psi$ is an isomorphism in codimension one, there are closed subsets $Z\subset \widetilde{X}$ and $W\subset \widetilde{Y}$ of codimension at least two such that $\psi$ restricts to an isomorphism $\widetilde{X}\setminus Z \xrightarrow{\cong}\widetilde{Y}\setminus W.$ Choose a divisor $T'\in |T|$ whose preimage on $\widetilde{X}$ contains no irreducible component of $Z$. Denote by $T'_{\widetilde{X}}, T'_{\widetilde{Y}}$ and $T'_Y$ the strict transforms of $T'$ on $\widetilde{X},\widetilde{Y}$ and $Y,$ respectively. Define $B'=B+T'-\frac{1}{m}\sum_{i=1}^mT_i$ and $B_Y'=B_Y+T'_Y-\frac{1}{m}\sum_{i=1}^m\phi_*T_i,$ and note that $(X,B')$ and $(Y,B_Y')$ are log Calabi--Yau pairs of complexity zero which are crepant via $\phi.$\par 
The variety $T'\cong \pp^{n-1}$ is a smooth variety of Picard rank one, and the pair $(T'.B_{T'})$ induced on $T'$ by $(X,B')$ via adjunction is a terminal log Calabi--Yau pair. It follows that every crepant birational equivalence $(T',B_{T'})\simeq_{\rm bir}(V,B_V)$ with $B_V$ effective is induced by an isomorphism $T'\cong V$ of underlying varieties. In particular, the birational maps in the commutative diagram above induce isomorphisms between $T', T'_{\widetilde{X}}, T'_{\widetilde{Y}}$ and $T'_Y.$ Note that it follows from this that $T'_{\widetilde{Y}}\cap W$ has codimension at least two in $T'_{\widetilde{Y}}.$\par
Since $f$ restricts to an isomorphism between neighborhoods of $T'$ and $T'_{\widetilde{X}}$, it follows that the normal bundle to $T'_{\widetilde{X}}$ in $\widetilde{X}$ is isomorphic to $\mathcal{O}_{\pp^{n-1}}(d).$ Since $Z, T'_{\widetilde{X}}\cap Z, W$ and $T'_{\widetilde{Y}}\cap W$ have codimension at least two in $\widetilde{X}, T'_{\widetilde{X}}, \widetilde{Y}$ and $T'_{\widetilde{Y}}$, respectively, it follows that the normal bundle to $T'_{\widetilde{X}}$ in $\widetilde{X}$ is isomorphic to $\mathcal{O}_{\pp^{n-1}}(d)$ as well. Since $T'_{Y}$ is the image of $T'_{\widetilde{Y}}$ under a morphism, it now follows that the normal bundle to $T'_Y$ in $Y$ is isomorphic to $\mathcal{O}_{\pp^{n-1}}(d+e)$ for some $e\geq 0.$ \par
We have just shown, for any choice of integers $d\geq 1,n\geq 2$ and $m\geq 3,$ that if the pair $(X,B)$ is crepant to a log Calabi--Yau pair $(Y,B_Y)$ of complexity zero then any toric log Calabi--Yau pair $(Y,\Delta)$ has a torus-invariant divisor $D\subset \Delta$ satisfying $D\cong \pp^{n-1}$ and $D\cdot C\geq d$ for all curves $C\subset D.$ So suppose that $X_1,\hdots, X_r$ is some finite set of $n$-folds which support log Calabi--Yau pairs of index $m$ and complexity zero. Given a toric log Calabi--Yau pair $(X_i,\Delta)$ supported on one of these $X_i$ and an irreducible component $D\subset\Delta_i,$ write
$$d_D=\inf\{D\cdot C\mid C\subset D\text{ a curve}\}.$$
Then write
$$d_{\rm max}=\max_D\{d_D\}.$$ 
Since we are considering only finitely many varieties $X_1,\hdots, X_r,$ it follows from Lemma ~\ref{lem:conjugate-tori} that this maximum exists and is finite. Choosing $d>d_{\rm max},$ we see that $(X,B)$ cannot be crepant to any log Calabi--Yau pair of complexity zero supported on a variety among $X_1,\hdots, X_r.$
}  
\end{example}
Example ~\ref{ex:inf-many-models} shows that one needs to consider infinitely many distinct generalized Bott towers in each dimension in order to account for all crepant birational equivalence classes of log Calabi--Yau pairs of complexity zero. However, it remains unclear precisely which generalized Bott towers are needed.
\begin{question}\label{que:minimal-sets-of-reps}
    Fix a positive integer $n\geq 2.$ Can one provide a description of a set $\mathcal{S}_n$ of generalized Bott towers with the property that for each log Calabi--Yau $n$-fold $(X,B)$ of birational complexity zero, there is exactly one $Y\in \mathcal{S}_n$ supporting a log Calabi--Yau pair $(Y,B_Y)$ of complexity zero crepant to $X$?
\end{question}

\bibliographystyle{habbvr}
\bibliography{references}

\end{document}